\documentclass{daj}

\usepackage{cite}
\usepackage{amssymb, amsmath,  amsthm, enumerate,  mathtools}

\usepackage{color}

\usepackage{graphicx}
\usepackage{tikz}

\usepackage{subcaption} 
\usepackage{caption}

\usepackage[thinlines]{easytable}

\usepackage{bm}
\usepackage{tikz}
\usepackage{tikz-3dplot}

\usetikzlibrary{calc,patterns,angles,quotes}

\usepackage{hyperref}
\usepackage{cite}

\usepackage{etoolbox}
\newtheorem{theorem}{Theorem}[section]
\newtheorem{lemma}[theorem]{Lemma}
\newtheorem{proposition}[theorem]{Proposition}

\newtheorem{definition}[theorem]{Definition}
\newtheorem{claim}[theorem]{Claim}
\newtheorem{remark}[theorem]{Remark}
\newtheorem{example}[theorem]{Example}

\dajAUTHORdetails{
  title = {Joints formed by lines and a $k$-plane, and a discrete estimate of Kakeya type}, 
  author = {Anthony Carbery, Marina Iliopoulou},
  plaintextauthor = {Anthony Carbery, Marina Iliopoulou},
  runningtitle = {Joints and Kakeya},
}   

\dajEDITORdetails{%
   year={2020},
   number={18},
   received={21 November 2019},   
   revised={22 July 2020},    
   published={29 December 2020},  
   doi={10.19086/da.18361},       
}   

\begin{document}

\begin{frontmatter}[classification=text]

\title{Joints formed by lines and a $k$-plane, and a discrete estimate of Kakeya type\footnote{This material is partly based upon work supported by the National Science Foundation under
Grant No. 1440140, while the authors were in residence at the Harmonic Analysis programme at the Mathematical Sciences
Research Institute in Berkeley, California, U.S.A., during the spring of 2017. The authors were further supported by LMS Research in Pairs Grant Ref 41870.}} 

\author[pgom]{Anthony Carbery}
\author[johan]{Marina Iliopoulou}

\begin{abstract}
Let $\mathcal{L}$ be a family of lines and let $\mathcal{P}$ be a family of $k$-planes in $\mathbb{F}^n$ where $\mathbb{F}$ is a field. In our first result we show that the number of joints formed by a $k$-plane in $\mathcal{P}$ together with $(n-k)$ lines in $\mathcal{L}$ is $O_n(|\mathcal{L}||\mathcal{P}|^{1/(n-k)}$). This is the first sharp result for joints involving higher-dimensional affine subspaces, and it holds in the setting of arbitrary fields $\mathbb{F}$. In contrast, for our second result, we work in the three-dimensional Euclidean space $\mathbb{R}^3$, and we establish the Kakeya-type estimate
\begin{equation*}\sum_{x \in J} \left(\sum_{\ell \in \mathcal{L}} \chi_\ell(x)\right)^{3/2} \lesssim |\mathcal{L}|^{3/2}\end{equation*}
where $J$ is the set of joints formed by $\mathcal{L}$; such an estimate fails in the setting of arbitrary fields. This result strengthens the known estimates for joints, including those counting multiplicities. Additionally, our techniques yield significant structural information on quasi-extremisers for this inequality.
\end{abstract}
\end{frontmatter}

\section{Introduction}
Let $\mathbb{F}$ be an arbitrary field and let $\mathcal{L}$ be a finite family of lines in $\mathbb{F}^n$ where $n\geq 3$. A {\bf joint} for $\mathcal{L}$ is a point $x \in \mathbb{F}^n$ at which $n$ lines from $\mathcal{L}$ with linearly independent directions meet. Denoting the set of joints by $J$, it has been proved (see especially \cite{Quilodran, KSS} and also \cite{Guth_Katz_2008, Guth_Katz_2010, Tao, Carbery_Iliopoulou_14, Chazelle_Edelsbrunner_Guibas_Pollack_Seidel_Sharir_Snoeyink_1992, EKS}) that 
$$ |J| \lesssim |\mathcal{L}|^{n/(n-1)}$$
where the implicit constant depends only on the dimension $n$, and in particular is independent of the field $\mathbb{F}$. Simple grid-like examples illustrate the optimality of the exponent $n/(n-1)$.

This result does not measure the extent to which joints can occur in a multiple fashion. For $x \in \mathbb{F}^n$
let 
\begin{equation*}
    N(x)= \#\{(l_1, \dots , l_n) \in \mathcal{L}^n \, : \, l_1, \dots , l_n \mbox{ form a joint at } x\}.
\end{equation*}
Following earlier works by Iliopoulou and by Hablicsek (see \cite{Iliopoulou_13, Iliopoulou_12, Iliopoulou_14, Hablicsek_14}), Zhang \cite{Zhang_16} has proved that 
\begin{equation}\label{counting joints with multiplicities}
    \sum_{x \in \mathbb{F}^n} N(x)^{1/(n-1)} \lesssim |\mathcal{L}|^{n/(n-1)}
\end{equation}
where once again the implicit constant depends only on the dimension.

A variant of this set-up is to consider the situation where we have $n$ families of lines $\mathcal{L}_1, \dots, \mathcal{L}_n$ of possibly very different cardinalities. Let 
$$N'(x) = \#\{(l_1, \dots , l_n) \in \mathcal{L}_1 \times \dots \times  \mathcal{L}_n\, : \, l_1, \dots , l_n \mbox{ form a joint at } x\}.$$
A point $x$ at which $N'(x) \neq 0$ is called a {\bf multijoint} for $\mathcal{L}_1, \dots, \mathcal{L}_n$. Zhang \cite{Zhang_16} has proved that 
$$ \sum_{x \in \mathbb{F}^n} N'(x)^{1/(n-1)} \lesssim |\mathcal{L}_1|^{1/(n-1)}\dots|\mathcal{L}_n|^{1/(n-1)},$$
which is formally stronger than (but is in fact equivalent to) the corresponding estimate when all the families of lines coincide. We refer to this result as the multijoints with multiplicities estimate. Once again, there was previous work of Iliopoulou on this problem (see \cite{Iliopoulou_12, Iliopoulou_13, Iliopoulou_14}) prior to Zhang's result.

Indeed, in the special case of $\mathbb{R}^3$, this multijoints with multiplicities estimate has
been proved via two different approaches, one in \cite{Iliopoulou_13}, where
the topology of $\mathbb{R}$ is exploited, and, as previously mentioned, another in \cite{Zhang_16}.
The goal of this paper is to present two new results, each one of which stems from one of the two approaches which have been hitherto developed. 

\textbf{Multijoints.} The first of these results relates to the approach in \cite{Zhang_16} and it gives a
small, but perhaps promising, step towards counting joints formed by
higher dimensional planes (rather than lines) in
$\mathbb{F}^n$, where $\mathbb{F}$ is an arbitrary
  field. This result is presented in Theorem~\ref{multijoints} and was announced in \cite{OWR}; we believe it to be the first
sharp result for joints and multijoints outside the setting of lines.

We describe the setting for this result. For $ 1 \leq j \leq d$, let $\mathcal{P}_j$ be a set of $k_j$ planes in $\mathbb{F}^n$, where $k_1 + \dots + k_d = n$. 
A {\bf multijoint} for $\{\mathcal{P}_j\}$ is a point of intersection of $d$ planes $P_j$, where $P_j \in \mathcal{P}_j$, such that if $\boldsymbol{\omega}_j$ is a set of 
vectors spanning the linear subspace parallel to $P_j$, then $\bigcup_{j=1}^d \boldsymbol{\omega}_j$ spans $\mathbb{F}^n$. Letting $J$ be the set of multijoints of $\{\mathcal{P}_j\}$, it is conjectured that\footnote{While we were preparing the final version of this paper for publication in Discret. Anal., the next two conjectures were solved by Tidor, Yu and Zhao. See \url{arXiv:2008.01610}.}
$$ |J| \lesssim |\mathcal{P}_1|^{1/(d-1)}\dots|\mathcal{P}_d|^{1/(d-1)},$$
and moreover that 
$$ \sum_{x \in \mathbb{F}^n} N'(x)^{1/(d-1)} \lesssim
|\mathcal{P}_1|^{1/(d-1)}\dots|\mathcal{P}_d|^{1/(d-1)},$$
where now 
$$N'(x) = \#\{(P_1, \dots , P_d) \in \mathcal{P}_1 \times \dots \times  \mathcal{P}_d\, : \, P_1, \dots , P_d \mbox{ form a multijoint at } x\}.$$
It is easy to see that the exponents $1/(d-1)$ are sharp. In Theorem~\ref{multijoints} we establish the first of these conjectures when all but one of the families $\mathcal{P}_j$ consists of a comparable number of lines.
Yang \cite{Yang} deals with the general setting, but an $\epsilon$-loss in the exponents is incurred. As we prepared this paper for publication, we were informed by Yu and Zhao that they have recently also obtained Theorem~\ref{multijoints} by somewhat different methods; see \cite{Yu_Zhao_19}. 

\textbf{Discrete Kakeya and quasi-extremals.} Wolff \cite{Wolff} first popularised the joints problem as a discrete analogue of the famous Kakeya problem and the corresponding Kakeya maximal problem. A strict analogue of the Kakeya maximal problem in the setting of arbitrary fields would involve bounding expressions of the form
$$ \sum_{x \in \mathbb{F}^n} \left(\sum_{l \in \mathcal{L}} \chi_l(x)\right)^{n/(n-1)}$$
by a quantity such as $|\mathcal{L}|^{n/(n-1)}$, under some hypothesis on $\mathcal{L}$ such as its members having distinct directions. In the setting of finite fields this sort of problem has been considered by Ellenberg, Oberlin and Tao \cite{EOT}. One cannot hope to have such an estimate in the case of infinite fields since the previously displayed expression will be infinite as soon as $\mathcal{L}$ is nonempty.
On the other hand, if one modifies the expression to include the sum only over the joints of $\mathcal{L}$, and thus to exclude certain lower-dimensional pathologies, it does indeed make sense to ask whether one has
\begin{equation*}\sum_{x \in J}\left(\sum_{l \in \mathcal{L}} \chi_l(x)\right)^{n/(n-1)} \lesssim |\mathcal{L}|^{n/(n-1)}
\end{equation*}
under the hypothesis that the family $\mathcal{L}$ consists of distinct lines (without imposing the condition that the members of $\mathcal{L}$ have distinct directions).

Note that, for a joint $x$, $\left(\sum_{l \in \mathcal{L}} \chi_l(x)\right)^{n}$ is at least as large as $N(x)$, and it may be significantly larger -- for example in $\mathbb{R}^3$, take $M \gg 1$ distinct coplanar lines through $0$ augmented by a further line through $0$ which is not in the common plane. The proposed estimate is therefore rather strong (stronger than \eqref{counting joints with multiplicities}): in fact, it fails in the setting of finite fields. (Indeed, consider the finite field $\mathbb{F}_p$, and take the family of all lines in $\mathbb{F}_p^2\times\{0\}$ together with one `vertical' line in
$\mathbb{F}_p^3$ passing through each point of $\mathbb{F}_p^2\times\{0\}$. Then we have a family $\mathcal{L}$ of $\sim p^2$ lines in $\mathbb{F}_p^3$ such that for each of $\sim p^2$ joints in $\mathbb{F}_p^3$, $\sum_{l \in \mathcal{L}} \chi_l(x) \sim p$, showing that the proposed estimate cannot hold in this setting.)

Our second new result establishes the proposed estimate in three-dimensional Euclidean space. In particular, further development of the approach to the multijoints with multiplicities problem in \cite{Iliopoulou_13} leads to the proposed estimate
\begin{equation*}
    \sum_{x\in J}\left(\sum_{l\in\mathcal{L}}\chi_l(x)\right)^{3/2}\lesssim |\mathcal{L}|^{3/2}
\end{equation*}
for an arbitrary family $\mathcal{L}$ of distinct lines in $\mathbb{R}^3$, and moreover it
provides a context for revealing the structure of
quasi-extremal configurations in this setting
(see Theorems~\ref{3d_basic} and \ref{3d}). Needless to say, our approach relies upon topological properties of Euclidean space which are not available in the setting of finite fields.

\bigskip

\textbf{Notation.} Before we proceed to state the main results, we establish some notation and terminology. If $A$ and $B$ are nonnegative quantities, we use the expression $A \sim B$ to denote the existence of absolute constants $c_n$ and $C_n$, whose precise values may vary from line to line as appropriate, such that $c_n B \leq A \leq C_n B$. We take $A \lesssim B$ to denote the existence of an absolute constant $C_n$, whose precise value may vary from line to line as appropriate, such that $A \leq C_n B$. We define $\gtrsim$ similarly. For a finite set $X$ we use the notations $\# X$ and $|X|$ interchangeably to denote its cardinality. A {\em definite proportion} of a finite set $X$ is a subset $X' \subseteq X$ such that $\#X' \gtrsim \# X$.

1.1. \textbf{Statement of results.} The first theorem concerns multijoints.
\begin{theorem}\emph{\textbf{(Multijoints estimate)}}\label{multijoints} Let $n\geq 3$ and $k\geq 2$. Let $\mathcal{L}_1,\ldots,\mathcal{L}_{n-k}$ be finite families of lines and $\mathcal{P}$ be a family of $k$-planes in $\mathbb{F}^n$. Let $J$ be the set of multijoints formed by these collections. Then,
\begin{equation*}
    |J|\lesssim L |\mathcal{P}|^{\frac{1}{d-1}},
\end{equation*}
where 
\begin{equation*}
    L:=\max\{|\mathcal{L}_1|,\ldots,|\mathcal{L}_{n-k}|\}
\end{equation*} and 
\begin{equation*}
   d:=n-k+1
\end{equation*}
denotes the total number of collections.
\end{theorem}

As we mentioned above, simple examples demonstrate the sharpness of the exponents in this result.

For our other main result, we first need a definition regarding structure of a set $J$ of points incident to a set $\mathcal{L}$ of lines. We will call this structure \textit{planar}. In particular, planar structure will imply that for each $x\in J$ there exists some special plane through $x$ which carries a definite proportion of the lines in $\mathcal{L}$ passing through $x$. Actually, we shall require something stronger:

\begin{definition}\label{definition: planar structure}{\rm Let $\mathcal{L}$ be a finite family of distinct lines in $\mathbb{R}^3$, and $J$ a set of points incident to lines in $\mathcal{L}$. We say that $J$ \textit{has planar structure} if there exist 
a set $\mathcal{P}$ of planes in $\mathbb{R}^3$,
and a partition of $J$ into pairwise disjoint sets $J_\Pi$, indexed by $\Pi\in\mathcal{P}$, such that
\begin{equation*}
    J_\Pi\subseteq \Pi \mbox{ for all }\Pi,
\end{equation*}
and so that the sets
\begin{equation*}
    \mathcal{L}_\Pi:=\{l\in\mathcal{L}:l\subseteq \Pi\text{ and }l\text{ contains some point in } J_\Pi \}
\end{equation*}
satisfy the following properties:
\begin{itemize}
    \item[P1)] For all $\Pi\in\mathcal{P}$, for all $x \in J_\Pi$,
    \begin{equation*}
       \#\{\text{lines in }\mathcal{L}_{\Pi}\text{ through }x\}\sim \#\{\text{lines in }\mathcal{L}\text{ through }x\}\footnote{The attentive reader will recognise that the definition of planar structure depends on the small implicit constant $c_1$ in this expression: we should therefore, strictly speaking, refer to this as $c_1$-planar structure.};
    \end{equation*}
    \item[P2)] The sets $\mathcal{L}_{\Pi}$, for $\Pi\in\mathcal{P}$, are pairwise disjoint.
   \end{itemize}
}

\end{definition}

\begin{remark}{\rm 
Further implications of planar structure are explored in Section~\ref{ppp}.
For now, observe that when $J$ has planar structure, the disjointness of the families $\mathcal{L}_{\Pi}$ implies that, in order to count incidences between $J$ and $\mathcal{L}$, it suffices to count incidences between $J_\Pi$ and lines in $\mathcal{L}_{\Pi}$ for each $\Pi\in\mathcal{P}$, and to then add the contributions from the different planes $\Pi$. This observation is relevant in particular in the proof of Lemma~\ref{lemma:nearlyplanarest} below.}
\end{remark}

\begin{example}{\rm Consider a Loomis--Whitney grid of joints at lattice points in $\mathbb{R}^3$, with one line parallel to each coordinate axis through each joint. Let $\mathcal{P}$ consist of the horizontal planes, and for $\Pi \in \mathcal{P}$ let $J_\Pi$ be the set of joints on $\Pi$. Then $\mathcal{L}_\Pi$ consists of those lines of $\mathcal{L}$ which lie in $\Pi$, and properties P1) and P2) are clear. 
On the other hand, a bush configuration through a single joint does not in general endow it with a planar structure, since there may be many more lines through the joint than are contained in any plane through it.}
\end{example}

For our purposes, a slightly weaker notion of planar structure is required. Informally, we will say that $J$ has nearly planar structure if there is an appropriate refinement of it which captures most of the incidences with $\mathcal{L}$, and which has planar structure. More precisely:

\begin{definition}
{\rm Let $\mathcal{L}$ be a finite family of distinct lines in $\mathbb{R}^3$, and $J$ a set of points incident to lines in $\mathcal{L}$. We say that $J$ has \emph{nearly planar structure} if for every dyadic $k \in \mathbb{N}$ there exists a subset $J_k'$ of
\begin{equation*}
    J_k:=\{x\in J:\;x\text{ lies in at least }k\text{ and fewer than }2k\text{ lines in }\mathcal{L}\}
\end{equation*}
so that 
\begin{equation*}
    |J_k'|\sim |J_k|\text{ for all }k
\end{equation*}
and
\begin{equation*}
    \bigcup_{k}J_k'\text{ has planar structure}\footnote{By this we mean that $|J_k'|\geq c_2|J_k|$ for all $k$, and that $\bigcup_{k}J_k'$ has $c_1$-planar structure, for specific constants $c_1,c_2$. Therefore, this definition should be thought of as describing $(c_1,c_2)$-nearly planar structure.}.
\end{equation*}
}
\end{definition}

 The reason why nearly planar structure is important to us is two-fold. Firstly, it gives the correct concept for analysing quasi-extremals for the proposed Kakeya inequality. Secondly, under the hypothesis of nearly planar structure, the validity of the Kakeya inequality can be established directly, see the key Lemma~\ref{lemma:nearlyplanarest} below. This, combined with 
the quasi-extremal analysis, then allows us to deduce the desired Kakeya inequality in the general setting.

\begin{theorem}{\emph{\textbf{(Discrete Kakeya-type theorem)}}}\label{3d_basic} For any finite set $\mathcal{L}$ of $L$ distinct lines in $\mathbb{R}^3$, the set $J$ of joints formed by $\mathcal{L}$ satisfies
\begin{equation}\label{hab}
   \sum_{x\in J}  \left(\sum_{l\in\mathcal{L}}\chi_l(x)\right)^{3/2}\lesssim L^{3/2}.
\end{equation}
Moreover, for any $0<\epsilon< 1/2$, the set $\tilde{J}$ of joints in $J$, each of which lies in $\lesssim L^{1/2}$ lines in $\mathcal{L}$, satisfies
\begin{equation*}
    \tilde{J}=J_{{\rm good}}\sqcup J_{{\rm bad}},
\end{equation*}
where $J_{\rm good}$ satisfies the exceptionally good estimate 
\begin{equation*}
    \sum_{x\in J_{{\rm good}}}\left( \sum_{l\in\mathcal{L}}\chi_l(x)\right)^{2-\epsilon}\lesssim_\epsilon L^{3/2}
\end{equation*}
and
\begin{equation*}
    J_{{\rm bad}}\text{ has nearly planar structure}.\footnote{By this we mean that $J_{{\rm bad}}$ has $(c_1, c_2)$-nearly planar structure for $c_1$ and $c_2$ absolute constants whose values depend on the constants in various well-established inequalities; see the discussion below Definition~\ref{def} in  Section~\ref{Section 5} for more details.}
\end{equation*}
\end{theorem}
We give a more detailed version of this result, including a more precise structural description of the sets $J_{\rm good}$ and $J_{\rm bad}$, in Section~\ref{Section 5} below. See Theorem~\ref{3d}.

The main thrust of the argument to prove estimate \eqref{hab} 
is to identify sufficient nearly planar structure 
for Lemma~\ref{lemma:nearlyplanarest} to apply. Indeed, we emphasise that the structural statement is the key point in Theorem~\ref{3d_basic}; it is the main new perspective that we offer, and most of the hard work goes into obtaining it.


\begin{remark}
{\rm Theorem~\ref{3d_basic} gives a structural analysis for sets of joints passing through $\lesssim L^{1/2}$ lines of $\mathcal{L}$. If, on the other hand, we consider sets of joints passing through $\gtrsim L^{1/2}$ lines of $\mathcal{L}$, it is not hard to see that they must be arranged in essentially non-interacting bushes. See Remark~\ref{large_l} below for more details. 
}
\end{remark}

\begin{remark}
{\rm The analysis of quasi-extremals implicit in Theorem~\ref{3d_basic} applies in particular in the setting of the joints problem \eqref{counting joints with multiplicities}.}
\end{remark}
\begin{remark}
{\rm We expect the Kakeya estimate \eqref{hab} to continue to hold when the real field is replaced by any field of characteristic zero. We also expect an $n$-dimensional analogue of \eqref{hab} (with exponents $3/2$ replaced by $n/(n-1)$) to hold at least in the case of the real field.}
\end{remark}


1.2. \textbf{Outline of the paper.} In common with many other results on joints, Theorems~\ref{multijoints} and \ref{3d_basic} are proved using the polynomial method. 

In particular, the multijoints Theorem~\ref{multijoints} is proved with the use of a polynomial that vanishes to appropriate order at the multijoints in question, and whose existence follows via a parameter counting argument. To carry this out in the case of arbitrary fields requires some of the machinery of Hasse derivatives of polynomials; to avoid disrupting the exposition, this ancillary material is postponed to Appendix~\ref{Appendix}. In Section~\ref{Section 2} we give an outline of the scheme of the proof and summarise 
the required polynomial calculus in the setting of the real field, where it is somewhat more straightforward. Then in Section~\ref{Section 3}
we complete the proof of Theorem~\ref{multijoints}.

The discrete Kakeya-type Theorem~\ref{3d_basic} will instead be proved using polynomial partitioning, which is described in Section~\ref{Section 4}. A more detailed discussion of the notion of planar structure, along with an extended outline of the proof of Theorem~\ref{3d_basic}, also features in Section~\ref{Section 4}. The details of the proof of Theorem~\ref{3d_basic} are given in Section~\ref{Section 5}.

\section{Preliminaries for the multijoints Theorem~\ref{multijoints}} \label{Section 2}

\subsection{Scheme of the proof} To motivate the discussion in this section, we briefly illustrate the main idea for the proof of the multijoints Theorem~\ref{multijoints} in $\mathbb{R}^n$, when $L\sim |\mathcal{P}|$. In this case (which in hindsight will be simpler from a technical perspective), the desired inequality becomes $|J|\lesssim |\mathcal{P}|^{\frac{d}{d-1}}$,where $J$ is the set of multijoints formed by the $\sim |\mathcal{P}|$ lines and $k$-planes in question, and where $d=n-k+1$. This situation is depicted in Figure~\ref{fig: multijoints}(a) in the special case where $n=5$ and $k=2$.

\begin{figure}[htbp]
\centering 
 
  %
  %
\begin{subfigure}[t]{0.45\textwidth}
\tdplotsetmaincoords{70}{110}
\begin{tikzpicture}[tdplot_main_coords,font=\sffamily, scale=0.7];

\draw[fill=yellow,opacity=0.4] (-4,-4,0) -- (-4,4,0) -- (4,4,0) -- (4,-4,0) -- cycle;
\draw[black, thin] (-4,-4,0) -- (-4,4,0) -- (4,4,0) -- (4,-4,0) -- cycle;

\draw[thick] (1.7,-1.5,-1.3) -- (1.3,-1.5,1.4);
\draw[blue, thick] (1.5,-2,-1.2) -- (1.5,-1,1.2);
\draw[magenta, thick] (1.5,-1.3,-1.2) -- (1.5,-1.7,1.2);
\fill[red] (1.5,-1.5,0) circle (1.1mm);

\draw[magenta, thick] (1.2,2.8,-1.7) -- (0.8,2.8,1.7);
\draw[thick] (1,2.3,-1.3) -- (1,3.3,1.3);
\draw[blue, thick] (1,3,-1.5) -- (1,2.6,1.5);
\fill[red] (1,2.8,0) circle (1.1mm);

\draw[blue, very thick] (0,0,4.7) -- (0,0,-4.7);

\draw[magenta, very thick] (-4,-4,-6) -- (4,4,6);

\draw[fill=yellow,opacity=0.4] (4,0.15,4) -- (4,0.15,6) -- (6,-0.15,6) -- (6,-0.15,4) -- cycle;
\draw[black, thin] (4,0.15,4) -- (4,0.15,6) -- (6,-0.15,6) -- (6,-0.15,4) -- cycle;
\draw[thick] (6.5,0,6.5) -- (-6,0,-6);
\draw[magenta, thick] (5,-1,5) -- (5,1,5);
\draw[blue, very thick] (3,-2,5) -- (7,2,5);
\fill[red] (5,0,5) circle (1.1mm);

\fill[red] (0,0,0) circle (1.1mm);

\draw[fill=yellow,opacity=0.4] (-2,-2.5,-5) -- (-4,-2.5,-5) -- (-4,-3.5,-4) -- (-2,-3.5,-4) -- cycle;
\draw[black, thin] (-2,-2.5,-5) -- (-4,-2.5,-5) -- (-4,-3.5,-4) -- (-2,-3.5,-4) -- cycle;
\draw[thick] (-3,-3,-6) -- (-3,-3,-3);
\draw[blue, thick] (-3,-4,-2.5) -- (-3,-2,-6.5);
\fill[red] (-3,-3,-4.5) circle (1.1mm);

\end{tikzpicture}
\caption{A multijoints configuration.}
\end{subfigure}
\begin{subfigure}[t]{0.45\textwidth}
\tdplotsetmaincoords{70}{110}
\begin{tikzpicture}[tdplot_main_coords,font=\sffamily, scale=0.7]

\draw[fill=yellow,opacity=0.4] (-4,-4,0) -- (-4,4,0) -- (4,4,0) -- (4,-4,0) -- cycle;
\draw[black, thin] (-4,-4,0) -- (-4,4,0) -- (4,4,0) -- (4,-4,0) -- cycle;

\node[anchor=south west,align=center] (line) at (3,4,-1) {$P\in\mathcal{P}$};
;

\fill[red] (1.5,-1.5,0) circle (1.1mm);
\draw[red, thick] (3,0,0) -- (0,-3,0);

\fill[red] (1,2.8,0) circle (1.1mm);
\draw[red, thick] (2,3.8,0) -- (-1,0.8,0);

\draw[blue, very thick] (0,0,4.7) -- (0,0,-4.7);

\node[color=blue, anchor=south west,align=center] (line) at (3,1.2,5.5) {$\ell_2$};
;

\draw[magenta, very thick] (-4,-4,-6) -- (4,4,6);

\node[color=magenta, anchor=south west,align=center] (line) at (3,3.8,5) {$\ell_3$};
;

\draw[thick] (6.5,0,6.5) -- (-6,0,-6);
\node[anchor=south west,align=center] (line) at (3,-1.2,5) {$\ell_1$};
\fill[red] (5,0,5) circle (1.1mm);
\node[anchor=south west,align=center] (line) at (0,0.2,-0.2) {$x$};

\fill[red] (0,0,0) circle (1.1mm);
\draw[red, thick] (2,2,0) -- (-2,-2,0);

\fill[red] (-3,-3,-4.5) circle (1.1mm);
\end{tikzpicture}
\caption{Core idea: each $x\in J$ is controlled via the vanishing properties of (derivatives of) $p$ on one of the objects $P$, $\ell_1$, $\ell_2$, $\ell_3$ through $x$.}
\end{subfigure}
 \captionsetup{singlelinecheck=off}
\caption[.]{\small{Diagram (a) features a configuration of a set $J$ of multijoints (drawn red) in $\mathbb{R}^5$, formed by sets $\mathcal{L}_1, \mathcal{L}_2, \mathcal{L}_3$ of (respectively black, blue and pink) lines and a set $\mathcal{P}$ of (yellow) 2-dimensional planes. Diagram (b) depicts the proof idea. For each $P\in\mathcal{P}$ and each multijoint on $P$, we draw a distinct red line through the multijoint, lying inside $P$. If we can count the red lines, we can count the multijoints. To that end, we find a low-degree polynomial $p$ vanishing identically on all the red lines (and thus on $J$). Fixing $P\in\mathcal{P}$, either $p_{|_P}\not\equiv 0$ or $p_{|_P}\equiv 0$. The former case is easy: we can control the number of multijoints on $P$, as the number of red lines on $P$ is at most $\deg p$. In the latter case, it transpires that for each multijoint $x\in P$ there exists $\ell\in\mathcal{L}_1\cup\mathcal{L}_2\cup\mathcal{L}_3$ through $x$ such that $\mathcal{D}_\ell\;\! p$ vanishes at $x$ but not identically on $\ell$. Therefore, the number of multijoints arising from such harder cases is at most the total number of roots of the polynomials $\mathcal{D}_\ell\;\! p_{|_\ell}$, over all $\ell\in\mathcal{L}_1\cup\mathcal{L}_2\cup\mathcal{L}_3$.}}\label{fig: multijoints}
\end{figure}
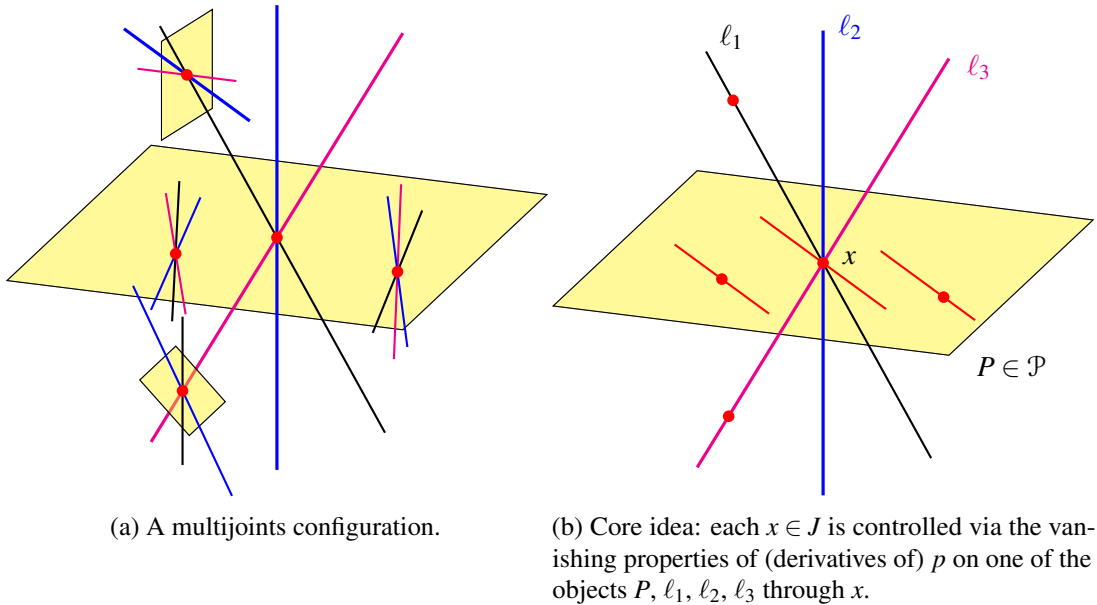

One would ideally wish to show that either each line or each plane contains $\lesssim |J|^{\frac{1}{d}}=|J|^{\frac{1}{n-(k-1)}}$ elements of $J$. This does not necessarily hold; however, we shall show that it is morally the case. 

Indeed, for each $P\in\mathcal{P}$, through each $x\in J\cap P$ we draw a distinct $(k-1)$-dimensional plane $\Pi$ lying inside $P\in\mathcal{P}$, as in Figure~\ref{fig: multijoints}(b). We thus draw $|J|$ such $(k-1)$-planes $\Pi$ in total. Parameter counting shows that there exists a non-zero $p\in\mathbb{R}[x_1,\ldots,x_n]$, with $\deg p\lesssim |J|^{\frac{1}{n-(k-1)}}$, which vanishes identically on each $\Pi$ (and thus also vanishes on $J$). 

Observe that if $p$ does not vanish identically on some $P\in\mathcal{P}$, then we automatically deduce that $P$ contains $\leq \deg p\lesssim |J|^{\frac{1}{n-(k-1)}}$ planes $\Pi$, and thus $\lesssim |J|^{\frac{1}{n-(k-1)}}$ multijoints, as desired.

For any $P\in\mathcal{P}$ on which $p$ vanishes identically, the above argument fails. However, we shall count the multijoints contained in such $k$-planes $P$ using the vanishing properties of appropriate \textit{directional derivatives} of $p$ on the lines in $\mathcal{L}_1\cup\ldots\cup\mathcal{L}_{n-k}$. This dichotomy is inspired by ideas in \cite{Zhang_16}
in which multijoints formed by lines are considered;
the directional derivatives we consider in fact already appear, in disguised form, in \cite{Zhang_16}.

In particular, for any $\ell\in\mathcal{L}_1\cup\ldots\cup\mathcal{L}_{n-k}$ denote by $\mathcal{D}_\ell \;\! p$ a derivative of $p$ of minimal order that does not vanish identically on $\ell$. Fixing $x\in J$ and $P\in\mathcal{P}$, $\ell_j\in\mathcal{L}_j$ through $x$ that together span $\mathbb{R}^n$, we will prove that if $p_{|_P}\equiv 0$, then there exists some $j\in\{1,\ldots,n-k\}$ such that $\mathcal{D}_{\ell_j}p$ vanishes at $x$ (but of course does not vanish identically on $\ell_j$, by definition). Therefore, the multijoints which were not counted by the earlier argument can be counted as roots of at most $|\mathcal{L}_1\cup\ldots\cup\mathcal{L}_n|\sim |\mathcal{P}|$ non-zero one-variable polynomials (the polynomials $\mathcal{D}_\ell\;\!p_{|_\ell}$, over all $\ell\in\mathcal{L}_1\cup\ldots\cup\mathcal{L}_n$), each of degree $\lesssim |J|^{\frac{1}{n-(k-1)}}$.

This perspective naturally motivates the study of directional derivatives of polynomials that vanish on a set of multijoints but do not vanish identically on all the planes/lines forming the multijoints. Subsection~\ref{polcalc} identifies such derivatives in the Euclidean setting. The more tedious generalisations to arbitrary field settings may be found in the Appendix.

\subsection{The zero polynomial} We begin with an elementary observation concerning zero polynomials and vanishing. 

\begin{definition} {\rm Let $R$ be a ring and $n\geq 1$. A polynomial $p\in R[x_1,\ldots,x_n]$ is \emph{the zero polynomial}, denoted by $p=0$, if all the coefficients of $p$ equal $0\in R$.}
\end{definition}

We can also think of $p\in R[x_1,\ldots,x_n]$ as its corresponding evaluation map $p:R^n\rightarrow R$. We often use the term \textit{polynomial mapping} to refer to an evaluation map. If $G \subseteq R^n$, then we will take $p_{|_G}\equiv 0$ to mean that $p(x) = 0$ for all $x \in G$. If $\mathbb{F}$ is a finite field, then there exist non-zero polynomials in $\mathbb{F}[x_1,\ldots,x_n]$ whose evaluation maps vanish identically on $\mathbb{F}^n$. For example, when $\mathbb{F}$ is a finite field of characteristic $q$, the non-zero polynomial $x^q -x$ vanishes everywhere. However, this cannot happen for infinite fields.

\begin{lemma} \label{infinite} Let $\mathbb{F}$ be an infinite field and $n\geq 1$. Then, for any $p\in\mathbb{F}[x_1,\ldots,x_n]$, $p$ is the zero polynomial if and only if $p$ vanishes everywhere on $\mathbb{F}^n$.

\end{lemma}

\begin{proof} It holds that any non-zero polynomial $f\in R[x]$, where $R$ is a commutative integral domain, has at most as many roots as its degree; therefore, if $R$ is infinite, there exists $r\in R$ such that $f(r)\neq 0\in R$. 

The above implies in particular that the statement of the lemma is true when $n=1$. Now, let $n\geq 2$ and let $p\in\mathbb{F}[x_1,\ldots,x_n]$ be non-zero. It follows that $p$ is a non-zero polynomial in $R[x_n]$, where $R=\mathbb{F}[x_1,\ldots,x_{n-1}]$ is an infinite commutative integral domain. Therefore, there exists $g\in R$ such that $p(x_1,\ldots,x_{n-1},g(x_1,\ldots,x_{n-1}))$ is a non-zero element of $R$, i.e. a non-zero polynomial in $\mathbb{F}[x_1,\ldots,x_{n-1}]$. By induction on $n$, it may be assumed that there exists $y\in\mathbb{F}^n$ such that $p(y,g(y))\neq 0\in\mathbb{F}$; that is, $p$ does not vanish at $(y,g(y))\in\mathbb{F}^n$.
\end{proof}

In order to prove Theorem~\ref{multijoints} we will work in the context of algebraically closed fields, which are always infinite. In such settings, the following corollary of Lemma~\ref{infinite} holds.

\begin{lemma} \label{fundamental theorem of algebra} Let $\mathbb{F}$ be an infinite field, $n\geq 1$ and let $p\in\mathbb{F}[x_1,\ldots,x_n]$ be a non-zero polynomial. If $\mathcal{P}$ is a family of distinct $(n-1)$-dimensional planes in $\mathbb{F}^n$ such that  $p_{|_\Pi}\equiv 0$ for every $\Pi\in\mathcal{P}$, then $|\mathcal{P}|\leq \deg p$.
\end{lemma}

\begin{proof}
Since $\mathbb{F}$ is an infinite field (and under the harmless assumption that $\mathcal{P}$ is a finite family), for every $x\in\mathbb{F}^n$ there exists a line $\ell$ in $\mathbb{F}^n$ through $x$ that intersects all the members of $\mathcal{P}$ at distinct points. Let $e(\ell)\in\mathbb{F}^n\setminus\{0\}$ be parallel to $\ell$. Assume that $|\mathcal{P}|> \deg p$; then the polynomial $p_{|_\ell}(t):=p(x+e(\ell)t)\in\mathbb{F}[t]$ has more roots than its degree, and is thus the zero polynomial. It follows in particular that $p(x)=0$. Since $x\in\mathbb{F}^n$ was arbitrary, $p$ vanishes everywhere on $\mathbb{F}^n$ and is hence the zero polynomial by Lemma~\ref{infinite}.
\end{proof}

\subsection{Polynomial calculus over the real field.}\label{polcalc}
We need to develop some of the calculus of polynomials in so far as it relates to multiplicities and restrictions to $k$-planes. We shall need to do so in arbitrary fields. It turns out that one may do this much more directly in the case of $\mathbb{R}$ than in that of an arbitrary field. This is partly because we can then use calculus freely, and partly because in this case we can avoid having to make a careful distinction between a polynomial as a member of $\mathbb{R}[x_1, \dots , x_n]$ and corresponding evaluations of it. For these reasons, we restrict ourselves for this subsection to the case of the real field: in the Appendix we develop the results in the case of arbitrary fields, via the Hasse calculus. Classical calculus is invariant under rigid but not affine motions: with a view to the development of the theory in arbitrary fields, we shall therefore want to focus on notions of the calculus of polynomials which are also affine-invariant, such as degree, vanishing and multiplicity.

Denote by $\mathbb{N}$ the set of nonnegative integers, that is, $\mathbb{N} = \{0,1,2, \dots \}$. Let $n\geq 1$. For any multiindex $a=(a_1,\ldots,a_n)\in\mathbb{N}^n$, let
\begin{equation*}
    |a|:=a_1+\cdots+a_n
\end{equation*}
be the \textit{length} of $a$. 

Let $p\in\mathbb{R}[x_1,\ldots,x_n]$, which we also consider as a polynomial mapping ${p}: \mathbb{R}^n \to \mathbb{R}$. If $p \neq 0$ we define the degree of the mapping $p:\mathbb{R}^n\rightarrow\mathbb{R}$ to be the least
$m$ such that $D^a p = 0$ for all $a$ with $|a| > m$. This notion of degree coincides with that arising by regarding $p$ as a member of $\mathbb{R}[x_1, \dots , x_n]$. The class of all polynomial mappings ${p}: \mathbb{R}^n \to \mathbb{R}$ of degree at most $d$ is a real vector space of dimension $\binom{n+d}{d} \sim_n d^n$. In this subsection, we focus on polynomial mappings, rather than polynomials, in order to be able to use calculus freely.

Let $\boldsymbol{\omega} = \{\omega_1 , \dots, \omega_n\}$ be a basis for $\mathbb{R}^n$. We denote the $(a, \boldsymbol{\omega})$ directional derivative of $p$ by
$$\mathcal{D}^a_{\boldsymbol{\omega}}p = (\boldsymbol{\omega} \cdot \nabla)^{a} p: = (\omega_1 \cdot \nabla)^{a_1} \cdots (\omega_n \cdot \nabla)^{a_n} p.$$
The degree of a non-zero $p$ is equivalently the least $m$ such that $\mathcal{D}^a_{\boldsymbol{\omega}}p=0$ for all $a$ with $|a| > m$, for {\em any} basis $\boldsymbol{\omega}$. If $A: \mathbb{R}^n \to \mathbb{R}^n$ is an affine map then the degree of $p \circ A$ coincides with that of $p$.

The {\bf multiplicity} of a non-zero $p$ at $y_0 \in \mathbb{R}^n$, ${\rm mult} (p, y_0)$, is the largest $m \in \mathbb{N}$ such that $({\mathcal{D}^a_{\boldsymbol{\omega}}p})(y_0) = 0$ for all $a$ with $a_1 + \cdots + a_n < m$. {\em This quantity is independent of the particular choice of basis $\boldsymbol{\omega}$ employed, and is invariant under affine maps of $\mathbb{R}^n$.}

Let $1 \leq k \leq n$ and let $x_0 \in \mathbb{R}^n$. Let $P=P(x_0, \omega_1, \dots , \omega_k)$ be the affine $k$-plane through $x_0$ which is parallel to the subspace spanned by $\{\omega_1 , \dots, \omega_k\}$. The same $P$ can arise as a $P(y_0, \nu_1, \dots , \nu_k)$ in many different ways.

The restriction $p_{|_P}$ of $p$ to ${P}$ is also a polynomial mapping (with $P$ identified with $\mathbb{R}^k$). Thus, for $z_0 \in P$, ${\rm mult}\left(p_{|_P}, z_0 \right)$ is canonically defined (independently of the base point $x_0$ or the particular vectors $\{\omega_1, \dots , \omega_k\}$ whose span, together with $x_0$, determines $P$).

In what follows we use the notation above.
\begin{lemma}[cf. Lemma~\ref{restriction properties}]\label{restriction properties_baby}
{\rm (i)} If for some $y_0 \in P$ 
\begin{equation*}(\omega_1 \cdot \nabla)^{a_1} \cdots (\omega_n \cdot \nabla)^{a_n} p (y_0) \neq 0,
\end{equation*}
then 
\begin{equation*}(\omega_{k+1} \cdot \nabla)^{a_{k+1}} \cdots (\omega_n \cdot \nabla)^{a_n} p_{|_P} \not\equiv 0.
\end{equation*}
\medskip
{\rm (ii)} For all $y \in P$ we have
$${\rm mult} \left((\omega_{k+1} \cdot \nabla)^{a_{k+1}} \cdots (\omega_n \cdot \nabla)^{a_n} p_{|_P},y\right) \geq {\rm mult}(p,y) - (a_{k+1} + \cdots + a_n).$$ 
\end{lemma}

\begin{proof} (i) This is clear because $(\omega_1 \cdot \nabla)^{a_1} \cdots (\omega_n \cdot \nabla)^{a_n} p$ is the result of applying the differential operator
$$(\omega_1 \cdot \nabla)^{a_1} \cdots (\omega_k \cdot \nabla)^{a_k}$$
to the function
$$(\omega_{k+1} \cdot \nabla)^{a_{k+1}} \cdots (\omega_n \cdot \nabla)^{a_n} p.$$
If the latter function is zero when restricted to $P$, any directional derivative of it in a direction parallel to $P$ will be zero when evaluated at any point of $P$. (Note that this argument breaks down in the case of arbitrary fields.)

\medskip
(ii)
Continuing, we also have that if $y \in P$,
and if $(a_1',\ldots,a_k')\in\mathbb{N}^k$ satisfies 
\begin{eqnarray*}
      (a_1'+\ldots+a_k')+ (a_{k+1}+\cdots+a_n) < 
       {\rm mult}(p,y),  
\end{eqnarray*}
then
$$(\omega_1 \cdot \nabla)^{a_1'} \cdots (\omega_k \cdot \nabla)^{a_k'} (\omega_{k+1}\cdot\nabla)^{a_{k+1}}\cdots (\omega_n\cdot\nabla)^{a_n}p(y)=0,$$
and so
\begin{equation*}
  (\omega_1 \cdot \nabla)^{a_1'} \cdots (\omega_k \cdot \nabla)^{a_k'} \big[(\omega_{k+1}\cdot\nabla)^{a_{k+1}}\cdots (\omega_n\cdot\nabla)^{a_n}p_{|_{P}}\big](y)=0,
  \end{equation*}
giving (ii). 
\end{proof}

\begin{lemma}[cf. Lemma~\ref{minimal derivative}]\label{minimal derivative_baby} Suppose that $p$ is non-zero.
Let $y_0 \in {P}$, and suppose that $a_1 + \cdots + a_n$ is minimal with respect to
$$(\omega_1 \cdot \nabla)^{a_1} \cdots (\omega_n \cdot \nabla)^{a_n} p (y_0) \neq 0.$$ 
Let $b_{k+1} + \cdots + b_n$ be minimal with respect to
$$(\omega_{k+1} \cdot \nabla)^{b_{k+1}} \cdots (\omega_n \cdot \nabla)^{b_n} p_{|_P} \not\equiv 0.$$
Then
$$ {\rm mult} \left((\omega_{k+1} \cdot \nabla)^{b_{k+1}} \cdots (\omega_n \cdot \nabla)^{b_n} p_{|_P}, y_0\right) \geq a_1 + \cdots + a_k.$$
\end{lemma}

\begin{proof}
It follows by Lemma~\ref{restriction properties_baby} that
\begin{equation*}
    (\omega_{k+1}\cdot\nabla)^{a_{k+1}}\cdots(\omega_n\cdot\nabla)^{a_n}p_{|_{P}}\not\equiv 0.
\end{equation*}
The minimality property of $(b_{k+1},\ldots,b_n)$ implies that $b_{k+1}+\cdots +b_n\leq a_{k+1}+\cdots +a_n$. On the other hand, the minimality property of $(a_{k+1},\ldots,a_n)$ means that $a_1+\cdots +a_n={\rm mult}(p,y)$, and so
\begin{eqnarray*}
    \begin{aligned}
       b_{k+1}+\ldots + b_n&\leq (a_1+\ldots+a_n)-(a_1+\cdots+a_k)\\
       &={\rm mult}(p,y_0)-(a_1+\cdots +a_k).
    \end{aligned}
\end{eqnarray*}
Combining this with assertion (ii) of Lemma~\ref{restriction properties_baby}, one deduces that
\begin{eqnarray*}
\begin{aligned}
{\rm mult} \left((\omega_{k+1} \cdot \nabla)^{b_{k+1}} \cdots (\omega_n \cdot \nabla)^{b_n} p_{|_P}, y_0\right)&\geq{\rm mult}(p,y_0)-(b_{k+1}+\cdots +b_n)\\
&\geq a_1+\cdots +a_k,
\end{aligned}
\end{eqnarray*}
as required.
\end{proof}

\begin{lemma}[cf. Lemma~\ref{order invariance of minimal derivative}]\label{order invariance of minimal derivative_baby} Suppose that $p$ is non-zero and that $\nu_{k+1}, \dots , \nu_{n} \, \in \mathbb{R}^n \setminus\{0\}$ are such that the set
$\{ \omega_1 , \dots, \omega_k,$ $\nu_{k+1}, \dots , \nu_{n}\}$ also forms a basis for $\mathbb{R}^n$. Let $b_{k+1} + \cdots + b_n$ be minimal with respect to
$$(\omega_{k+1} \cdot \nabla)^{b_{k+1}} \cdots (\omega_n \cdot \nabla)^{b_n} p_{|_P} \not\equiv 0,$$
and let  $c_{k+1} + \dots + c_n$ be minimal with respect to
$$(\nu_{k+1} \cdot \nabla)^{c_{k+1}} \cdots (\nu_n \cdot \nabla)^{c_n} p_{|_P} \not\equiv 0.$$
Then $ b_{k+1} + \cdots + b_n = c_{k+1} + \dots + c_n.$
\end{lemma}

\begin{proof} 
Suppose that for some $b \in \mathbb{N}$ we have
$$(\omega_{k+1} \cdot \nabla)^{\delta_{k+1}} \cdots (\omega_n \cdot \nabla)^{\delta_n} p_{|_P} \equiv 0$$
whenever $\delta_{k+1} + \cdots + \delta_n < b$. Fix $\beta_{k+1}, \dots , \beta_n$ with $\beta_{k+1} + \cdots + \beta_n < b$. It suffices to show that 
$$(\nu_{k+1} \cdot \nabla)^{\beta_{k+1}} \cdots (\nu_n \cdot \nabla)^{\beta_n} p_{|_P} \equiv 0.$$
Each $\nu_j$ is a linear combination of $\omega_r$'s, and multiplying out the expression 
$$(\nu_{k+1} \cdot \nabla)^{\beta_{k+1}} \cdots (\nu_n \cdot \nabla)^{\beta_n}$$
using the binomial theorem leads to a (weighted) sum of expressions of the form
$$ (\omega_1 \cdot \nabla)^{\gamma_1} \cdots (\omega_n \cdot \nabla)^{\gamma_n},$$
where $\gamma_1+\cdots+\gamma_n=\beta_{k+1}+\cdots+\beta_n<b$. Now
$$\prod_{j = k+1}^n(\omega_{j} \cdot \nabla)^{\gamma_{j}} p$$
vanishes on $P$ by hypothesis, and further derivatives of this expression in directions parallel to $P$ will continue to return zero. Summing, we conclude that 
$$(\nu_{k+1} \cdot \nabla)^{\beta_{k+1}} \cdots (\nu_n \cdot \nabla)^{\beta_n} p$$
vanishes on $P$, as required.
\end{proof}


Despite appearances to the contrary, it is not completely obvious how to generalise the above arguments to the case of arbitrary fields. For details of these natural extensions via the Hasse calculus, see the Appendix.

\section{Proof of Theorem~\ref{multijoints}} \label{Section 3}

\textbf{Theorem \ref{multijoints}.} \textit{Let $n\geq 3$ and $k\geq 2$. Let $\mathcal{L}_1,\ldots,\mathcal{L}_{n-k}$ be finite families of lines and $\mathcal{P}$ be a family of $k$-planes in $\mathbb{F}^n$. Let $J$ be the set of multijoints formed by these collections. Then,}
\begin{equation*}
    |J|\lesssim L |\mathcal{P}|^{\frac{1}{d-1}},
\end{equation*}
\textit{where }
\begin{equation*}
    L:=\max\{|\mathcal{L}_1|,\ldots,|\mathcal{L}_{n-k}|\}
\end{equation*} \textit{and }
\begin{equation*}
   d:=n-k+1
\end{equation*}
\textit{denotes the total number of collections.}

\begin{proof} It may be assumed that $\mathbb{F}$ is algebraically closed (and therefore infinite), since the lines and $k$-planes in the collections $\mathcal{L}_1,\ldots,\mathcal{L}_{n-k}$, $\mathcal{P}$ can be naturally extended to lines and $k$-planes in $\overline{\mathbb{F}}^n$ (where $\overline{\mathbb{F}}$ is the algebraic closure of $\mathbb{F}$),  still forming the multijoints in $J$.

For every multijoint $x$, fix lines $l_i(x)\in\mathcal{L}_i$, $i=1,\ldots,n-k$, and a $k$-plane $P(x)\in \mathcal{P}$ that form a multijoint at $x$. We say that $x$ {\em chooses} these lines and $k$-plane.

For every $k$-plane $P \in \mathcal{P}$, let $J_P$ be the set of multijoints that have chosen $P$; it holds that $J_P\subseteq P$. For some $B\in \mathbb{N}$ that will be fixed later, fix $\Pi_P$ to be a family of distinct $(k-1)$-planes contained in ${P}$, with exactly $B$ of them through each element of $J_P$, so that each $(k-1)$-plane in $\Pi_P$ contains exactly one multijoint in $J_P$. In particular,
$$|\Pi_P|=|J_P|B.
$$
Note that the existence of such distinct $(k-1)$-planes contained in $P$ is ensured by the condition $k \geq 2$ and the fact that $\mathbb{F}$ is algebraically closed and therefore infinite.

The goal is to count these $(k-1)$-planes contained in $P$; the above equality will then directly give an estimate on the number of multijoints in $P$. And, indeed, under certain conditions, the number of these $(k-1)$-planes contained in $P$ can be controlled, as they will all lie in the zero set of a relatively low degree polynomial that does not vanish identically on $P$. The existence of such a polynomial will follow from Claim~\ref{polynomial method} below, which uses a standard parameter-counting argument.

More precisely, for some large parameter $T>0$, fix natural numbers
\begin{equation*}
    A\sim \frac{T\cdot L}{{(L^{n-k}\;|\mathcal{P}|)^{1/d}|J|^{1/d}}}
\end{equation*}
and
\begin{equation*}
    B\sim \frac{T\cdot |\mathcal{P}|}{(L^{n-k}\;|\mathcal{P}|)^{1/d}|J|^{1/d}}.
\end{equation*}
For each $P\in \mathcal{P}$, fix $e_1(P),\ldots,e_{k}(P)\in\mathbb{F}^n$ which span $P$, and $e_{k+1}(P),\ldots,e_n(P)\in\mathbb{F}^n$ transverse to $P$ (see Definition \ref{def: transverse directions}). Claim~\ref{polynomial method} below states that there exists a low degree polynomial, all of whose derivatives in directions $e_{k+1}(P),\ldots,e_n(P)$ up to order $A$ vanish on all $(k-1)$-planes in $\Pi_P$, for all $P\in\mathcal{P}$.

\begin{claim}\label{polynomial method} For all $T>0$ sufficiently large, there exists non-zero $p\in\mathbb{F}[x_1,\ldots,x_n]$ with \begin{equation*}
    \deg p\lesssim T
\end{equation*}
such that for any $P\in\mathcal{P}$
\begin{equation*}
    \big(e_{k+1}(P)\cdot \nabla\big)^{\lambda_{k+1}}\cdots \big(e_n(P)\cdot \nabla\big)^{\lambda_n}p_{|_\Pi}\equiv 0
\end{equation*}
for all $\Pi\in \Pi_P$, for all $(\lambda_{k+1},\ldots,\lambda_n)\in\mathbb{N}^{n-k}$ with $\lambda_{k+1}+\cdots +\lambda_n\leq A$. 
\end{claim}

Here and below we are employing Hasse derivatives -- for more details see the Appendix.

Note that $\big(e_{k+1}(P)\cdot \nabla\big)^{\lambda_{k+1}}\cdots \big(e_n(P)\cdot \nabla\big)^{\lambda_n}p_{|_\Pi}$ above denotes the usual restriction to $\Pi$ of the function $\big(e_{k+1}(P)\cdot \nabla\big)^{\lambda_{k+1}}\cdots \big(e_n(P)\cdot \nabla\big)^{\lambda_n}p:\mathbb{F}^n\rightarrow\mathbb{F}$. Since $\mathbb{F}$ is an infinite field, this restriction is the zero function if and only if the polynomial
\begin{equation*}
    \big(e_{k+1}(P)\cdot \nabla\big)^{\lambda_{k+1}}\cdots \big(e_n(P)\cdot \nabla\big)^{\lambda_n}p(x_0+\Omega_{\Pi}t)=0\in\mathbb{F}[t_1,\ldots,t_{k-1}]
\end{equation*}
for any $x_0\in P$ and any $n\times (k-1)$ matrix $\Omega_{\Pi}$ whose columns are $(k-1)$ fixed linearly independent vectors in $\mathbb{F}^n$ parallel to $\Pi$.

\begin{proof}[Proof of Claim~\ref{polynomial method}] For each $\Pi\in \bigcup_{P\in\mathcal{P}}\Pi_P$, fix $(k-1)$ linearly independent vectors in $\mathbb{F}^n$ which are parallel to $\Pi$, and denote by $\Omega_{\Pi}$ the $n\times (k-1)$ matrix with these vectors as columns. Recall that $\Pi_P$ is the disjoint union, over all $x\in J$ that have chosen $P$ (i.e., with $P(x)=P$), of all $\Pi\in\Pi_{P(x)}$ through $x$. Therefore, we may take our polynomial to be any non-zero $p\in\mathbb{F}[x_1,\ldots,x_n]$ with $\deg p\lesssim T$ such that, for any $x\in J$,
\begin{equation}\label{eq:pointwise}
    \big(e_{k+1}(P)\cdot \nabla\big)^{\lambda_{k+1}}\cdots \big(e_n(P)\cdot \nabla\big)^{\lambda_n}p(x+\Omega_\Pi t)=0
\end{equation}
in $\mathbb{F}[t_1,\ldots,t_{k-1}]$, for all $\Pi\in \Pi_{P(x)}$ through $x$, for all $(\lambda_{k+1},\ldots,\lambda_{n})\in\mathbb{N}^{n-k}$ with $\lambda_{k+1}+\cdots +\lambda_{n}\leq A$. 

Now, we assert that in order to ensure that a polynomial $p$ of degree at most $D$ satisfies the vanishing requirements above, it suffices to impose $\sim |J|BA^{n-k}D^{k-1}$ linear conditions on the coefficients of the polynomial. Indeed, for each $x\in J$, for each one of the $B$ in total $(k-1)$-planes $\Pi\in\Pi_{P(x)}$ through $x$, we simply require that each of the $\sim A^{n-k}$ polynomials
\begin{equation*}
    \big(e_{k+1}(P)\cdot \nabla\big)^{\lambda_{k+1}}\cdots \big(e_n(P)\cdot \nabla\big)^{\lambda_n}p(x+\Omega_{\Pi}t)\in\mathbb{F}[t_1,\ldots,t_{k-1}],
\end{equation*}
for all $\lambda_{k+1}+\cdots +\lambda_n\leq A$, is the zero polynomial. Since $\mathbb{F}$ is an infinite field, each of these polynomials is the zero polynomial in $\mathbb{F}[t_1,\ldots,t_{k-1}]$ as long as it vanishes with multiplicity at least $D+1$ at $0$ along each of $(D+1)^{k-2}$ lines through $0$ appropriately arranged in $\mathbb{F}^{k-1}$. (To see this, first consider the case $k=3$, and then proceed by induction.) Therefore, a non-zero polynomial of degree $\leq D$ with the desired vanishing properties exists as long as 
\begin{equation*}
    |J|A^{n-k}BD^{k-1}\lesssim D^n,
\end{equation*}
or equivalently
\begin{equation*}
    |J|A^{n-k}B\lesssim D^{n-k+1}=D^{d},
\end{equation*}
a property that is satisfied by the chosen parameters when $D\sim T$.
\end{proof}

Fix $p$ as in Claim~\ref{polynomial method}. We say that a $k$-plane $P\in\mathcal{P}$ is \textit{exceptional} if
\begin{equation*}
    \big(e_{k+1}(P)\cdot \nabla\big)^{\lambda_{k+1}}\cdots \big(e_n(P)\cdot \nabla\big)^{\lambda_n}p_{|_P}\not\equiv 0
\end{equation*}
for some $(\lambda_{k+1},\ldots,\lambda_{n})\in\mathbb{N}^{n-k}$ with $\lambda_{k+1}+\cdots +\lambda_{n}\leq A$. Let 
\begin{equation*}
    J_{\text{exc}}:=\{x\in J:\; P(x)\text{ is exceptional}\}.
\end{equation*}
It will transpire that using Claim~ \ref{polynomial method} one can count the multijoints in $J_{\text{exc}}$. The main observation at this point is that the multijoints which cannot be counted using Zhang's argument in \cite{Zhang_16} are all in $J_{\text{exc}}$. 

More precisely, recall that for each $x\in J$ we have fixed lines $l_1(x)\in\mathcal{L}_1,\ldots, l_{n-k}(x)\in\mathcal{L}_{n-k}$ through $x$; denote by $e(l_1(x)),\ldots, e(l_{n-k}(x))$ their respective directions and observe that these directions are transverse to $P(x)$ since $x$ is a multijoint.

Since $p$ is not the zero polynomial, for every $x\in J$ there exists $a(x)=(a_1(x),\ldots,a_n(x))\in\mathbb{N}^n$ of minimal length such that
\begin{equation} \label{eq:lowest degree monomial}
   \big(e_1(P(x))\cdot\nabla\big)^{a_1(x)}\cdots \big(e_k\big(P(x)\big)\cdot \nabla)^{a_k(x)}\cdot\Big(e\big(l_1(x)\big)\cdot\nabla\Big)^{a_{k+1}(x)}\cdots\Big(e\big(l_{n-k}(x)\big)\cdot \nabla\Big)^{a_n(x)}p(x)\neq 0.
\end{equation}
We fix some choice of $\{a(x)\}_{x \in J}$. We say that $x$ \textit{is of type }1 if 
$$a_{k+1}(x)+\cdots + a_n(x)>A;
$$
otherwise, we say that $x$ \textit{is of type }2. 

Let $J_1$ be the set of multijoints in $J$ of type 1, and $J_2$ the set of multijoints in $J$ of type 2.

\textbf{Estimating $\boldsymbol{|J_1|}$.} The multijoints in $J_1$ can be counted in a similar manner as in \cite{Zhang_16}. Indeed, let $x\in J_{1}$. By definition, it holds that
$$a_{k+1}(x)+\cdots +a_{n}(x)>A,
$$
thus there exists $i\in\{1,\ldots,n-k\}$ for which $x$ \textit{is of type }$(1,i)$, meaning that
$$a_{k+i}(x)\gtrsim A.
$$
Fix $i\in\{1,\ldots,n-k\}$. Since $p$ is not the zero polynomial, it follows by Lemma~\ref{restriction properties_baby}/Lemma~\ref{restriction properties} that for every line $l\in\mathcal{L}_i$ there exists a directional derivative $\mathcal{D}_lp$ of $p$ of minimal order such that
\begin{equation*}
    \mathcal{D}_lp_{|_l}(t):=\mathcal{D}_l p\big(x_0+te(l)\big)\neq 0\in\mathbb{F}[t]
\end{equation*}
for some (any) $x_0\in l$.
By Lemma~\ref{minimal derivative_baby}/Lemma~\ref{minimal derivative} and the minimality property of $a(x)\in\mathbb{N}^n$ this derivative satisfies
\begin{equation*}
     {\rm mult}(\mathcal{D}_lp_{|_l},t_y)\geq a_{k+i}(y)\gtrsim A\text{ for all }y\in J\text{ of type }(i,1)\text{ which choose }l,
\end{equation*}
where for each $y\in l$, $t_y\in \mathbb{F}$ is defined by $y=x_0+t_ye(l)$. Thus, by B\'ezout's theorem,
\begin{eqnarray*}
    \begin{aligned}
        |\{x\in J\text{ of type }(i,1)\}| \, A&=\sum_{l\in\mathcal{L}_i}\;\;\;\sum_{x\in J\text{ of type }(i,1)\text{ choosing }l}A\\
        & \lesssim \sum_{l\in\mathcal{L}_i}\;\;\;\sum_{x\in J\text{ of type }(i,1)\text{ choosing }l}{\rm mult}(\mathcal{D}_lp_{|_l},t_x)\\
        &\leq \sum_{l\in\mathcal{L}_i}\deg \mathcal{D}_lp_{|_l}\leq \sum_{l\in\mathcal{L}_i}\deg p\\
        &\lesssim |\mathcal{L}_i| T.
    \end{aligned}
\end{eqnarray*}
It follows that
$$|\{x\in J\text{ of type }(i,1)\}|\cdot\frac{T L}{{(L^{n-k}\;|\mathcal{P}|)^{1/d}|J|^{1/d}}}\lesssim |\mathcal{L}_i|T,
$$
and thus
$$|\{x\in J\text{ of type }(i,1)\}|\lesssim (L^{n-k}\;|\mathcal{P}|)^{1/d}|J|^{1/d}
$$
for all $i=1,\ldots,n-k$, implying that
$$|J_1|\lesssim  (L^{n-k}\;|\mathcal{P}|)^{1/d}|J|^{1/d}.
$$

\textbf{Estimating $\boldsymbol{|J_2|}$.} The crucial observation here is that 
\begin{equation*}
    J_2\subseteq J_{\text{exc}}.
\end{equation*}
Indeed, let $x\in J_2$. By definition,
\begin{equation*}
    a_{k+1}(x)+\cdots+a_n(x)\leq A.
\end{equation*}
Combining \eqref{eq:lowest degree monomial} with Lemma~\ref{restriction properties_baby}/Lemma~\ref{restriction properties} and Lemma \ref{infinite}, one obtains
\begin{equation*}
    \Big(e\big(l_1(x)\big)\cdot\nabla\Big)^{a_{k+1}(x)}\cdots\Big(e\big(l_{n-k}(x)\big)\cdot \nabla\Big)^{a_n(x)}p_{|_{P(x)}}\not\equiv 0.
\end{equation*}

Lemma~\ref{order invariance of minimal derivative_baby}/Lemma~\ref{order invariance of minimal derivative} thus implies that any directional derivative $\mathcal{D}p$ of $p$ of minimal order with the property that $\mathcal{D}p_{|_{P(x)}}\not\equiv 0$ has order at most $A$. In particular, any derivative $\overline{\mathcal{D}}p$ of $p$ in directions $e_{k+1}(P(x)),\ldots,e_n(P(x))$ (the vectors appearing in the statement of Claim~\ref{polynomial method}) of minimal order such that 
\begin{equation} \label{small non-vanishing derivative}
   \overline{\mathcal{D}}p_{|_{P(x)}}\not\equiv 0
\end{equation}
takes the form
\begin{equation*}
    \overline{\mathcal{D}}p=\big(e_{k+1}(P(x))\cdot\nabla\big)^{\lambda_{k+1}}\cdots \big(e_n(P(x))\cdot\nabla\big)^{\lambda_n}p
\end{equation*}
for some $(\lambda_{k+1},\ldots,\lambda_n)\in\mathbb{N}^{n-k}$ with $\lambda_{k+1}+\cdots +\lambda_n\leq A$. Since the existence of such a derivative is guaranteed (see Remark~\ref{considering different directions} for further clarification), it immediately follows that $P(x)$ is exceptional, hence $x\in J_{\text{exc}}$. 

It thus suffices to estimate $|J_{\text{exc}}|$. Observe that
\begin{equation*}
    J_{\text{exc}}=\bigsqcup_{\text{exceptional }P\in\mathcal{P}}\;J_P.
\end{equation*}
Now, let $P\in\mathcal{P}$ be an exceptional $k$-plane; by definition, there exists $(\lambda_{k+1},\ldots,\lambda_{n})\in\mathbb{N}^{n-k}$, with $\lambda_{k+1}+\cdots+\lambda_n\leq A$, such that
\begin{equation*}
    \big(e_{k+1}(P)\cdot \nabla\big)^{\lambda_{k+1}}\cdots \big(e_n(P)\cdot \nabla\big)^{\lambda_n}p_{|_P}\not\equiv 0.
\end{equation*}
On the other hand, by Claim~\ref{polynomial method} it further holds that
\begin{equation*}
    \big(e_{k+1}(P)\cdot \nabla\big)^{\lambda_{k+1}}\cdots \big(e_n(P)\cdot \nabla\big)^{\lambda_n}p_{|_\Pi}\equiv 0\text{ for all }\Pi\in\Pi_P.
\end{equation*}

Therefore, the polynomial
\begin{equation*}
    g(t):=\big(e_{k+1}(P)\cdot \nabla\big)^{\lambda_{k+1}}\cdots \big(e_n(P)\cdot \nabla\big)^{\lambda_n}p(x_P+t_1e_1(P)+\cdots +t_ke_k(P))\in\mathbb{F}[t_1,\ldots,t_k]
\end{equation*}
(for some fixed $x_P\in P$) is not the zero polynomial, but it vanishes everywhere on $\widetilde{\Pi}$ for every $\widetilde{\Pi}$ in a family of distinct $(k-1)$-planes in $\mathbb{F}^k$ of size $|\Pi_P|$. It follows by Lemma~\ref{fundamental theorem of algebra} that
\begin{equation*}|J_P|B=|\Pi_P|\leq \deg g\lesssim T
\end{equation*}
for every exceptional $P$.
Therefore,
\begin{equation*}
    |J_{\text{exc}}|= \sum_{\text{exceptional }P\in\mathcal{P}}\;\;|J_P|\leq |\mathcal{P}|\, \max_P |J_P|\lesssim \frac{|\mathcal{P}|T}{B}
\end{equation*}
and thus
\begin{equation*}
    |J_2|\leq |J_{\text{exc}}|\lesssim (L^{n-k}\;|\mathcal{P}|)^{1/d}|J|^{1/d}.
\end{equation*}

Combining the above estimates on $|J_1|$ and $|J_2|$, one obtains the desired estimate
$$|J|\lesssim L|\mathcal{P}|^{\frac{1}{d-1}}.
$$

\end{proof}

\section{Preliminaries for the discrete Kakeya-type Theorem~\ref{3d_basic}} \label{Section 4}

In this section we further explain the statement of Theorem~\ref{3d_basic} and outline some computational estimates of an algebraic-geometric nature which are useful for its proof.

\subsection{Further understanding planar structure}\label{ppp} Let $J$ be a set of points, incident to lines in a family $\mathcal{L}$, that has planar structure.

As already mentioned in the Introduction, the planar structure of $J$ implies, roughly speaking, that there exists a plane through each point $x\in J$ that contains the bulk of the lines in $\mathcal{L}$ through $x$. Such a situation in itself however is not sufficient to imply planar structure.

In particular, bearing in mind the notation of Definition~\ref{definition: planar structure}, view the points in $J_\Pi$ and the lines in $\mathcal{L}_{\Pi}$ as \textit{associated to} $\Pi$. Assign a different colour to each plane $\Pi$, and assign the colour of $\Pi$ to the points in $J_\Pi$ and the lines in $\mathcal{L}_{\Pi}$. (Note that a blue plane may contain a red line, and a blue line may contain a red point.) We say that a \textit{fan} of colour $\mathcal{C}$ is any collection of coplanar lines of colour $\mathcal{C}$ all passing through the same point of colour $\mathcal{C}$ (which may be thought of as the \textit{root} of the fan, or the point from which the fan \textit{emanates}).

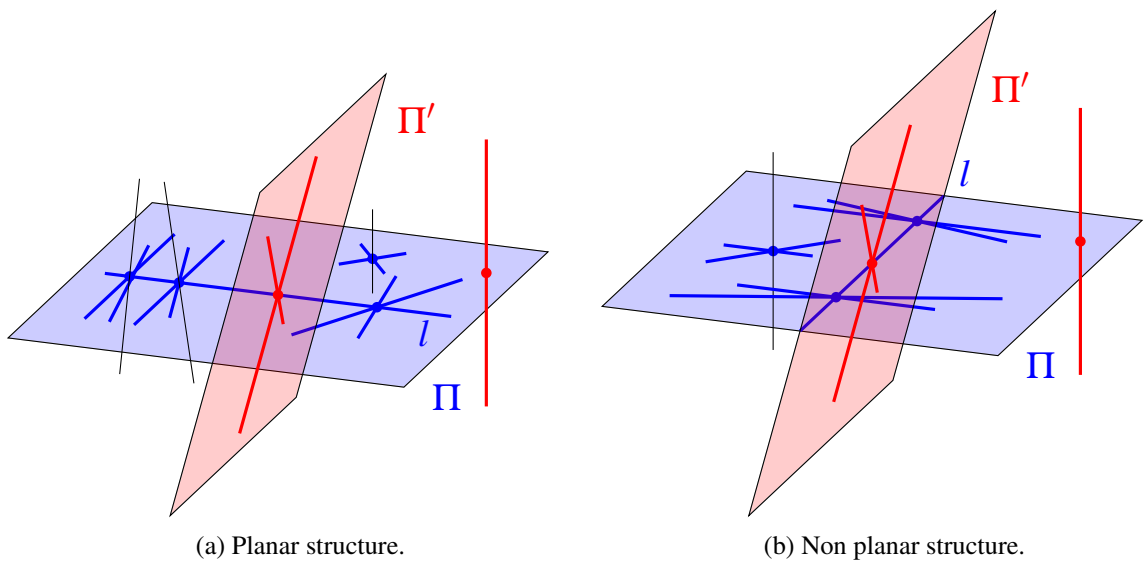
\begin{figure}[htbp]
\centering 
 
  %
  %
\begin{subfigure}[t]{0.49\linewidth}
\tdplotsetmaincoords{70}{110}
\begin{tikzpicture}[tdplot_main_coords,font=\sffamily, scale=0.7];

\node[anchor=south west,align=center] (line) at (3,4,-1) {\Large{\color{blue}$\Pi$}};
;

\node[anchor=south west,align=center] (line) at (2.4,3.5,0) {\Large{\color{blue}$l$}};
;

\draw[fill=blue,opacity=0.2] (-4,-4,0) -- (-4,4,0) -- (4,4,0) -- (4,-4,0) -- cycle;
\draw[black, thin] (-4,-4,0) -- (-4,4,0) -- (4,4,0) -- (4,-4,0) -- cycle;
 
\draw[very thick, blue](0,-3.5,0)--(0,3.5,0);

\draw[very thick, blue](2.5,-3,0)--(-2.5,-3,0);

\draw[very thick, blue](-1.75,-3.25,0)--(2.5,-2.5,0);

\draw[very thick, blue](2.5,-2,0)--(-2.5,-2,0);

\draw[very thick, blue](1.9,-1.5,0)--(-1.9,-2.5,0);

\draw[very thick, blue](1.75,2.25,0)--(-1.75,1.75,0);

\draw[very thick, blue](2,1,0)--(-2,3,0);

\fill[blue] (0,-3,0) circle (1mm);
\draw[black](0,-3.2,-2)--(0,-2.8,2);

\fill[blue] (0,-2,0) circle (1mm);
\draw[black](0,-2.3,2)--(0,-1.7,-2);

\fill[blue] (0,2,0) circle (1mm);

\fill[blue] (-2.5,1,0) circle (1mm);
\draw[very thick, blue](-3,1.5,0)--(-2,0.5,0);
\draw[very thick, blue](-1.8,1.5,0)--(-3.2,0.5,0);

\draw[black](-2.5,1,1)--(-2.5,1,-0.7);

\node[anchor=south west,align=center] (line) at (3,3.3,4.5)
{\Large${\color{red}\Pi'}$};

\fill[red] (-2.5,3.3,0) circle (1mm);
\draw[red, very thick] (-2.5,3.3,2.7) -- (-2.5,3.3,-2.7);

\tdplotsetrotatedcoords{-90}{75}{0}
\begin{scope}[tdplot_rotated_coords]
\draw[fill=red,opacity=0.2] (-3.5,-3.5,0) -- (-3.5,3.5,0) -- (3.5,3.5,0) -- (3.5,-3.5,0) -- cycle;
\draw[black, thin] (-3.5,-3.5,0) -- (-3.5,3.5,0) -- (3.5,3.5,0) -- (3.5,-3.5,0) -- cycle;

\fill[red] (0,0,0) circle (1mm);

\draw[red, very thick] (0,0,0);
\draw[red, very thick] (3,0,0) -- (-3,0,0);
\draw[red, very thick] (-2,2,0)--(1,-1,0);
\end{scope}
\end{tikzpicture}
\caption{Planar structure.}
\end{subfigure}
\begin{subfigure}[t]{0.49\linewidth}
\tdplotsetmaincoords{70}{110}
\begin{tikzpicture}[tdplot_main_coords,font=\sffamily, scale=0.7]
\draw[fill=blue,opacity=0.2] (-4,-4,0) -- (-4,4,0) -- (4,4,0) -- (4,-4,0) -- cycle;
\draw[black, thin] (-4,-4,0) -- (-4,4,0) -- (4,4,0) -- (4,-4,0) -- cycle;

\node[anchor=south west,align=center] (line) at (3,4,-1) {\Large{\color{blue}$\Pi$}};
;

\node[anchor=south west,align=center] (line) at (3,3.3,4.5) {\Large${\color{red}\Pi'}$};

\node[anchor=south west,align=center] (line) at (-4,0.1,0) {\Large${\color{blue}l}$};

\draw[very thick, blue](-4,0,0)--(4,0,0);

\fill[blue] (-2.5,0,0) circle (1mm);
\fill[blue] (2,0,0) circle (1mm);

\fill[blue] (0,-2,0) circle (1mm);
\draw[blue, very thick](0,-3,0)--(0,-1.2,0);
\draw[blue, very thick](1,-3,0)--(-1,-1,0);
\draw[black](0,-2,2)--(0,-2,-2);

\draw[blue, very thick] (2,-2,0) -- (2,2,0);
\draw[blue, very thick](3,-3,0)--(1,3,0);

\draw[blue, very thick] (-2.5,-2.5,0) -- (-2.5,2.5,0);
\draw[blue,very thick](-2,2,0)--(-3,-2,0);

\fill[red] (-2.5,3.3,0) circle (1mm);
\draw[red, very thick] (-2.5,3.3,2.7) -- (-2.5,3.3,-2.7);

\tdplotsetrotatedcoords{-90}{75}{0}
\begin{scope}[tdplot_rotated_coords]
\draw[fill=red,opacity=0.2] (-4,-4,0) -- (-4,4,0) -- (4,4,0) -- (4,-4,0) -- cycle;
\draw[black, thin] (-4,-4,0) -- (-4,4,0) -- (4,4,0) -- (4,-4,0) -- cycle;

\fill[red] (0,0,0) circle (1mm);

\draw[red, very thick] (0,0,0);
\draw[red, very thick] (3,0,0) -- (-3,0,0);
\draw[red, very thick] (-2,2,0)--(1,-1,0);

\end{scope}

\end{tikzpicture}
\caption{Non planar structure.}
\end{subfigure}
 \captionsetup{singlelinecheck=off}
\caption[.]{\small{Above are examples of an allowed and a forbidden configuration inside a set with planar structure. The black lines are lines in $\mathcal{L}$ that have not been coloured, i.e. not assigned to a plane. The second diagram demonstrates that a union of sets with planar structure does not necessarily have planar structure.}}\label{allowed}
\end{figure}

\begin{enumerate}[(i)]
   \item Property P1) implies that, if $\Pi$ is blue, then there is a blue fan inside $\Pi$ emanating from each blue point in $\Pi$. Moreover, roughly speaking, each such fan contains the bulk of lines in $\mathcal{L}$ through its root (as in Figure~\ref{allowed}(a).)
   \item If a blue plane $\Pi$ contains a red point $x$, then $x$ is associated to some red plane $\Pi'\neq \Pi$. In particular, there exists a red fan emanating from $x$ fully inside $\Pi'$ (as in Figure~\ref{allowed}(a).)
   \item Observe that property P2) can be rephrased as follows: Let $\Pi\in\mathcal{P}$ and $l\in\mathcal{L}_\Pi$; then
    \begin{equation*}
       \text{ for each }x\in J\cap l\text{, either }x\in J_\Pi\text{ or }x\in J_{\Pi'}\text{ for some }\Pi'\in\mathcal{P}\text{ transverse to }l.
\end{equation*}
To illustrate this, let $l$ be a blue line. This means that there exists a blue fan emanating from each blue point in $l$, lying fully inside the blue plane $\Pi$ that $l$ is associated to. If $l$ contains some red point $x$ as well, then the fact that $l$ is not red (by P2)) implies that the red plane $\Pi'$ associated to $x$ (which carries the red fan emanating from $x$) cannot contain $l$ (and is thus transverse to $l$); see Figure~\ref{allowed}(a). In other words, if $l$ is the intersection of the blue plane $\Pi$ with a red plane $\Pi'$, then $l$ does not contain red points (and thus there are no red fans, rooted at $l$, that live inside $\Pi'$). In other words, a configuration as in Figure~\ref{allowed}(b) is forbidden for a set of planar structure.
\end{enumerate}

\subsection{Algebraic preliminaries} Theorem \ref{3d_basic} will be proved using the \textit{polynomial partitioning} technique of Guth and Katz \cite{Guth_Katz_2010}. The method, described in the theorem that follows, exploits the topology of Euclidean space to partition finite sets of points into smaller parts, using the zero set of a polynomial.

\begin{theorem}\emph{\textbf{(Polynomial partitioning)}}\label{polynomial partitioning} Let $S$ be a finite set of points in $\mathbb{R}^n$, and $d>1$. Then there exists a non-zero polynomial $p \in \mathbb{R}[x_1,...,x_n]$, of degree $\leq d$, and $\sim_n d^n$ pairwise disjoint open sets (cells) $C_1,\ldots,C_m$, each of which contains $\leq |S|/m\lesssim_n |S|/d^n$ points of $S$, such that $\mathbb{R}^n=C_1\sqcup\ldots\sqcup C_m\sqcup Z_p$, where $Z_p$ is the zero set of $p$.\end{theorem}

Since its birth in 2010, polynomial partitioning has revolutionised incidence geometry, and has further shed light on some long-standing harmonic analytic problems. The reason is that, when it comes to point-line incidences, extremising situations tend to occur when the points and lines in question cluster on low-degree varieties. When this is indeed the case for a specific point-line incidence problem, polynomial partitioning has the potential to allow a reduction of the original problem to this type of situation. In other words, and roughly speaking, it naturally reduces to the study of extremisers. 

At a more technical level, polynomial partitioning may be viewed as a divide-and-conquer approach: the fact that each cell carries few points suggests that its contribution to point-line incidences could potentially be controlled by some induction argument. If that is achieved, it remains to control the point-line incidences that occur on the zero set itself -- and this is facilitated via the computational bounds below, which follow from B\'ezout's theorem in algebraic geometry.

\begin{theorem} \textbf{\emph{(Guth--Katz \cite{Guth_Katz_2008})}} \label{bezout for lines}Let $p_1,p_2\in\mathbb{R}[x_1,x_2,x_3]$ be non-zero. If $p_1, p_2$ do not have a common factor, then at most $\deg p_1 \cdot \deg p_2$ lines in $\mathbb{R}^3$ lie simultaneously in the zero set of $p_1$ and the zero set of $p_2$.
\end{theorem}

\begin{definition}
{\rm Let $p\in\mathbb{R}[x_1,x_2,x_3]$ be a non-zero polynomial of degree $\leq d$. Let $Z$ be the zero set of $p$. Denote by $p_{sf}$ the square-free polynomial obtained after eliminating all the squares appearing in the expression of $p$ as a product of irreducible polynomials in $\mathbb{R}[x_1,x_2, x_3]$.\footnote{Observe that $p$ and $p_{sf}$ have the same zero set.}

A \textit{critical point $x$ of $Z$} is a point of $Z$ for which $\nabla p_{sf}(x) = 0$. Any other point of $Z$ is
called a \textit{regular point of $Z$}. A point $x\in\mathbb{R}^3$ is a \textit{flat point of $Z$} if it is a regular point of $Z$ lying in at least three co-planar lines of $Z$.

A line in $\mathbb{R}^3$ a \textit{critical line of $Z$} if each point of the line is a critical point of $Z$. 

A line $l$ in $\mathbb{R}^3$ is a \textit{flat line of $Z$} if all the points of $l$, except perhaps for finitely many, are regular
points of $Z$ on which the second fundamental form of $Z$ vanishes.}
\end{definition}

Note that the tangent space to $Z$ at $x$ is well-defined at all regular points $x$ of $Z$.

The number of critical lines inside a variety can easily be controlled by Theorem \ref{bezout for lines} above.

\begin{proposition}\label{few critical lines} \textbf{\emph{(Guth--Katz \cite{Guth_Katz_2008})}}  Let $Z$ be the zero set of a non-zero $p\in \mathbb{R}[x_1, x_2, x_3]$. Then $Z$ contains at most $d^2$ critical lines.
\end{proposition}

Flat points of a variety are points where the second fundamental form of the variety vanishes. And these, in turn, are points where specific appropriate polynomials simultaneously vanish. Using this fact, Theorem~\ref{bezout for lines} also yields control on the number of flat lines inside a variety.

\begin{proposition}\label{few flat lines} \textbf{\emph{(Elekes--Kaplan--Sharir \cite{EKS})}} Let $Z$ be the zero set of a non-zero $p\in \mathbb{R}[x_1, x_2, x_3]$. If a line $l$ in $\mathbb{R}^3$ contains at least $3d-3$ flat points of $Z$, then $l$ is a flat line of $Z$. Moreover, at most $3d^2-4d$ flat lines do not fully lie inside planes contained in $Z$.
\end{proposition}

\subsection{The Szemer\'edi--Trotter theorem} The discrete Kakeya-type Theorem~ \ref{3d} is a statement on incidences between lines and joints in $\mathbb{R}^3$. To prove it, we shall use the following theorem (which, note, fails in general field settings).

\begin{theorem} \emph{\textbf{(Szemer\'edi--Trotter \cite{S-T})}} Let $S$ be a finite set of points in $\mathbb{R}^2$ and $\mathcal{L}$ a finite set of lines in $\mathbb{R}^2$. Then, if $I(S,\mathcal{L})$ denotes the number of incidences between $S$ and $\mathcal{L}$, it holds that
\begin{equation*}
    I(S,\mathcal{L})\lesssim |S|^{2/3}|\mathcal{L}|^{2/3}+|\mathcal{L}|+|S|.
\end{equation*}
In particular, for any $k \geq 2$, if $S_k$ denotes the set of points in $S$ each lying in at least $k$ and fewer than $2k$ lines of $\mathfrak{L}$, then
\begin{equation*}
    |S_k|\lesssim \frac{|\mathcal{L}|^2}{k^3}+\frac{|\mathcal{L}|}{k}.
\end{equation*}
\end{theorem}

\section{Proof of Theorem~\ref{3d_basic}} \label{Section 5}

As well as giving our desired discrete Kakeya estimate, Theorem~\ref{3d} below gives strong structural information on configurations of joints and lines that quasi-extremise discrete Kakeya-type inequalities.

 Let $\mathcal{L}$ be a finite family of $L$ distinct lines in $\mathbb{R}^3$, and $J$ a set of joints formed by $\mathcal{L}$. For dyadic $k \in \mathbb{N}$ let
\begin{equation*}
    J_k:=\{x\in J: \; x\text{ lies in at least }k\text{ and fewer than }2k\text{ lines in }\mathcal{L}\}.
\end{equation*}
For reasons that will become clear later, $k\sim L^{1/2}$ is a natural threshold for us. We refer to values of $k \leq c L^{1/2}$ as {\em small}
and values of $k >c L^{1/2}$ as {\em large} (any constant $c$ is allowed in this dichotomy, as long as it is fixed throughout the argument).
\begin{definition}\label{def}
     {\rm
   For each $0<\epsilon<1/2$ and each dyadic $k\lesssim L^{1/2}$ (i.e., each small dyadic $k$), we say that $k$ is \emph{good} if $J_k$ satisfies the exceptionally good estimate
     \begin{equation*}
         |J_k|k^{2-\epsilon/2}\lesssim_\epsilon L^{3/2};
     \end{equation*}
     otherwise, we say that $k$ is \emph{bad}.
     }
\end{definition}
Note that the notions of goodness and badness are $\epsilon$-dependent, but this will not be of concern to us, since $\epsilon$ will be fixed. The implicit constants in the definition of good $J_k$ should be thought of as large: they will need to be uniformly bounded below by absolute constants which will arise from the proof of Theorem~\ref{3d} -- for example the constant in the  Szemer\'edi--Trotter theorem, the constant in the joints inequality $|J|\lesssim |\mathcal{L}|^{3/2}$ in $\mathbb{R}^3$,  the constant in the simple multijoints inequality $|J|\lesssim (|\mathcal{L}_1||\mathcal{L}_2||\mathcal{L}_3|)^{1/2}$ in $\mathbb{R}^3$ (discussed in the Introduction), and constants arising as part of standard refinement processes when employing polynomial partitioning.

The result below gives precise structural information on the union of the ``bad" sets $J_k$, and asserts that they do not obstruct our desired strong discrete Kakeya inequality.

\begin{theorem}{\emph{\textbf{(Discrete Kakeya-type theorem)}}}\label{3d} For any finite set $\mathcal{L}$ of $L$ distinct lines in $\mathbb{R}^3$, the set $J$ of joints formed by $\mathcal{L}$ satisfies
\begin{equation}\label{eq:discrete kakeya'}
   \sum_{x\in J}  \left(\sum_{l\in\mathcal{L}}\chi_l(x)\right)^{3/2}\lesssim L^{3/2}.
\end{equation}
Moreover, for any $0<\epsilon< 1/2$, the set $\tilde{J}$ of joints in $J$, each of which lies in $\lesssim L^{1/2}$ lines in $\mathcal{L}$, may be decomposed as
\begin{equation*}
    \tilde{J}=J_{{\rm good}}\sqcup J_{{\rm bad}},
\end{equation*}
where $J_{\rm good}$ satisfies the exceptionally good estimate 
\begin{equation}\label{eq:excepgoood}
    \sum_{x\in J_{{\rm good}}}\left( \sum_{l\in\mathcal{L}}\chi_l(x)\right)^{2-\epsilon}\lesssim_\epsilon L^{3/2}
\end{equation}
and
\begin{equation*}
    J_{{\rm bad}}\text{ has nearly planar structure }.
\end{equation*}
In particular, we may take
\begin{equation*}
    J_{\rm good}:=\bigcup_{{\rm good \text{ }}k}J_k\text{ and }J_{\rm bad}:=\bigcup_{{\rm bad \text{ }}k}J_k.
\end{equation*}

\end{theorem}

\textbf{Structure of the proof of Theorem~\ref{3d}.}
The proof of Theorem~\ref{3d} is rather involved, and so we outline its six principal steps. We first fix $\epsilon\in (0,1/2)$. We make the preliminary observation that with $\mathcal{L}$, $J$ and $J_k$ as in the statement of the theorem, the desired inequality \eqref{eq:discrete kakeya'} becomes
\begin{equation}\label{eq:discrete kakeya}
\sum_k |J_k|k^{3/2}\lesssim L^{3/2}
\end{equation}
and the exceptionally good estimate  \eqref{eq:excepgoood} becomes 
\begin{equation}\label{eq:excepgood_k}
\sum_{\text{good } k} |J_k|k^{2- \epsilon}\lesssim_\epsilon L^{3/2},
\end{equation}
in which expressions, and in all to follow, only dyadic $k$ are considered. We observe that \eqref{eq:excepgood_k} is a direct consequence of the definition of goodness: we have
\begin{equation*}
  \sum_{\text{good }k}|J_k|k^{2-\epsilon}= \sum_{\text{good }k}\frac{|J_k|k^{2-\epsilon/2}}{k^{\epsilon/2}}\lesssim_\epsilon L^{3/2}\sum_k\frac{1}{k^{\epsilon/2}}\lesssim_\epsilon L^{3/2}.
\end{equation*}
Moreover, the estimate
\begin{equation*}
    \sum_{k\gtrsim L^{1/2}} |J_k|k^{3/2}\lesssim L^{3/2}
\end{equation*}
immediately follows from the Szemer\'edi--Trotter theorem (see Step 6 for details). Thus, in order to prove \eqref{eq:discrete kakeya} it suffices to consider bad $k$ only. In fact, as we have mentioned already, the key difficulty is obtaining the structural statement, and we focus on this in the first five steps of the proof. In the final Step 6 we use the structural statement to complete the proof of \eqref{eq:discrete kakeya}. 

To establish the structural statement we need to show that $J_{\text{bad}}=\bigcup_{\text{bad }k}J_k$ has nearly planar structure.

We shall first focus on the contributions to $J_{\text{bad}}$ coming from joints in each individual $J_k$; interactions between different values of $k$ come into play only when we begin to expose the planar structure in Step 5.


In \textbf{Step 1} we use polynomial partitioning with suitable parameters to partition $J_k$ (for bad $k$) into those points lying in a variety $Z_k$, and those lying in $\mathbb{R}^3 \, \setminus \, Z_k$. 

By the end of Step 4, we will have shown that $J_k$ has nearly planar structure; and, crucially, that the bulk of the lines incident to most of the joints in $J_k$ lie in planes inside $Z_k$. Then, in Step 5, the interaction between the various partitioning varieties $Z_k$ (corresponding to all bad $k$) will be studied in order to show that $\bigcup_{\text{bad }k} J_k$ has nearly planar structure.

In \textbf{Step 2} we specify the choice of parameters from Step 1 in order to ensure that a definite proportion of the points of $J_k$ in fact lie in $Z_k$, thus reducing matters to the ``algebraic" case. Steps 1 and 2 essentially feature as part of the analysis in \cite{Guth_Katz_2010}, and we do not claim any originality here.

In \textbf{Step 3} we begin to explore structures within the set $J_k$ of joints and the set of lines forming them. In particular, we show that a definite proportion $\bar{J}_k$ of the joints in $J_k$ are regular points of ${Z_k}$ and, crucially, live inside planes contained in ${Z_k}$. These planes 
are further shown to contain the bulk of the lines forming each joint in $\bar{J}_k$.

In \textbf{Step 4} we easily deduce that the set $\bar{J}_k$ identified in Step 3 has planar structure. Note that this implies that $J_k$ has nearly planar structure for every bad $k$.

\textbf{Step 5} is the central step in our analysis and represents the heart of the matter: showing that $J_{\text{bad}}=\bigcup_{\text{bad }k}J_k$ has nearly planar structure. While in general a union of sets with planar structure needs not have planar structure (see the discussion in Section~\ref{ppp}), we show in this step that the sets $\bar{J}_k$ in turn have subsets 
of definite proportion, whose corresponding union has planar structure. In other words, the set $\bigcup_{\text{bad } k}\bar{J}_k$ has nearly planar structure, and therefore $J_{\text{bad}} $ has nearly planar structure too. We begin the argument in Step 5a by describing potential obstructions to nearly planar structure. Such obstructions naturally motivate the study of the interaction of the various varieties $Z_k$ for bad $k$ (as these varieties, and more precisely the planes inside them, carry the joints in the various sets $\bar{J}_k$, and most of the lines forming them). This study is undertaken in Step 5b and allows us to establish the crucial Claim~\ref{interaction}. In particular, this claim allows us in Step 5c to assert that no potential obstruction to nearly planar structure can actually succeed.


Finally, in \textbf{Step 6}, we establish the free-standing Lemma~\ref{lemma:nearlyplanarest} which shows that inequality \eqref{eq:discrete kakeya} (and therefore inequality \eqref{eq:discrete kakeya'}) holds 
in the presence of nearly planar structure. We use this, together with \eqref{eq:excepgood_k}, and easy arguments for large $k$ (i.e., $k\gtrsim L^{1/2}$), to finally establish \eqref{eq:discrete kakeya} in the general case.

We now give the details.

\begin{proof} Let $\epsilon\in (0,1/2)$. For the first five steps of the proof, we focus only on (small) bad $k$.

\textbf{Step 1: Partitioning }$\boldsymbol{J_k}$ \textbf{for each bad} $\boldsymbol{k}.$ Since $k\lesssim L^{1/2}$, the Szemer\'edi--Trotter theorem asserts that
\begin{equation*}
    |J_k|\lesssim \frac{L^2}{k^3}.
\end{equation*}
Therefore, for an appropriately large constant $A>0$ (independent of $k$ and $\epsilon$) which will be specified in Step 2, the quantity 
\begin{equation*}
    d_k:=A\frac{L^2}{|J_k|k^3}
\end{equation*}
is larger than 1. It follows by the polynomial partitioning Theorem~ \ref{polynomial partitioning} that there exists a non-zero $p_k\in \mathbb{R}[x_1,x_2,x_3]$,
with $\deg p_k\leq d_k$, whose zero set $Z_k$ splits 
$\mathbb{R}^3$ in $\sim d_k^3$ cells, each containing $\lesssim \frac{|J_k|}{d_k^3}$ elements of $J_k$.

\textbf{Step 2: Reducing to the joints in $\boldsymbol{Z_k}$.} Either $\gtrsim |J_k|$ elements of $J_k$ lie in the union of the cells (the cellular case) or $\gtrsim |J_k|$ elements of $J_k$ lie in $Z_k$ (the algebraic case). However, the constant $A$ will be fixed to be large enough for the cellular case to be impossible; thus, the algebraic case will hold. 

More precisely, suppose that the cellular case holds. The following claim holds -- its proof is a standard counting argument, and is included here for purposes of  self-containment.

\begin{claim}In the cellular case, there exists a cell $C$ such that
\begin{equation}\label{eq:fair share}
    |J_k\cap C|\sim \frac{|J_k|}{d_k^3}\text{ and }|\mathcal{L}_C|\lesssim \frac{L}{d_k^2},
\end{equation}
where $\mathcal{L}_C$ is the set of lines in $\mathcal{L}$ that cross $C$.
\end{claim}

\begin{proof} In the cellular case, there is some absolute constant $0 <a <1$ such that at least $a|\mathcal{C}|$ of the cells $C$ satisfy 
\begin{equation*}
    |J_k\cap C|\lesssim \frac{|J_k|}{d_k^3}.
\end{equation*}
Indeed, by the polynomial partitioning Theorem~\ref{polynomial partitioning}, $|J_k\cap C|\leq |J_k|/|\mathcal{C}|$ for each cell $C$ (where $\mathcal{C}$ denotes the set of cells and is $\sim d_k^3$). Combining this with the fact that $\sum_C|J_k\cap C|\geq \bar{c} |J_k|$ for some absolute constant $\bar{c}$ (since the cellular case holds), one obtains that at least $a|\mathcal{C}|$ of the cells $C$ satisfy $|J_k\cap C|\geq c |J_k|/|\mathcal{C}|$, for an appropriately small constant $c$. Indeed, otherwise the cells that satisfy $|J_k\cap C|\geq c |J_k|/|\mathcal{C}|$ contribute fewer than $a|C|\cdot |J_k|/|\mathcal{C}|=a|J_k|$ joints in total, while the remaining cells contribute fewer than $|\mathcal{C}|\cdot c|J_k|/|\mathcal{C}|=c|J_k|$ joints in total. Therefore, the cells contribute fewer than $(a+c)|J_k|$ joints in total, which is a contradiction for appropriately small $a$ and $c$.

On the other hand, at least $(1-a)|\mathcal{C}|$ of the cells $C$ satisfy that
\begin{equation*}
    |\mathcal{L}_C|\leq H \frac{L}{d_k^2},
\end{equation*}
for some large absolute constant $H$. Indeed, if the above fails, then at least $a|\mathcal{C}|$ of the cells $C$ are each crossed by at least $H L/d_k^2$ lines in $\mathcal{L}$. Since $|\mathcal{C}|\sim d_k^3$, it follows that
\begin{equation*}
    \sum_{C\in\mathcal{C}}|\mathcal{L}_C|>100Ld_k.
\end{equation*}
This is a contradiction; indeed, a line $l$ in $\mathbb{R}^3$ cannot cross more than $d_k+1$ cells (as otherwise $l$ would intersect $Z_k$ more than $d_k$ times, and would thus lie in $Z_k$, which would imply that $l$ crosses 0 cells). Therefore, $\sum_{C\in\mathcal{C}}|\mathcal{L}_C|=\sum_{l\in\mathcal{L}}\#\{\text{cells that }l\text{ crosses}\}\leq L(d_k+1)\leq 2Ld_k$, contradicting the earlier estimate.

By pigeonholing, there exists a cell $C$ that satisfies the statement of the claim.
\end{proof}

Fix a cell $C$ that satisfies \eqref{eq:fair share}. It holds that $|J_k\cap C|>k$. Indeed, if $|J_k\cap C|\leq k$ then \eqref{eq:fair share} implies that $\frac{|J_k|}{d_k^3}\lesssim k$, or equivalently $|J_k|k^2\lesssim L^{3/2}$ (recalling the definition of $d_k$). This is a contradiction because $k$ is bad (harmlessly assuming that the implicit constant in the definition of bad $k$ is sufficiently large relative to the soon-to-be-specified constant $A$). Now, since $|J_k\cap C|>k$ and since each joint in $J_k\cap C$ has $\sim k$ lines in
$\mathcal{L}$ through it, there exist at least $k+(k-1)+(k-2)+\cdots +1\gtrsim k^2$ lines in $\mathcal{L}$ crossing the cell $C$. That is, $|\mathcal{L}_C|\gtrsim k^2$, or equivalently $k\lesssim |\mathcal{L}_C|^{1/2}$. 
Therefore, the Szemer\'edi--Trotter theorem applied to count incidences between $J_k\cap C$ and $\mathcal{L}_C$ gives that
\begin{equation*}
     |J_k\cap C|\lesssim \frac{|\mathcal{L}_C|^2}{k^3},
\end{equation*}
which, by the bounds \eqref{eq:fair share} on $|J_k\cap C|$ and $|\mathcal{L}_C|$, implies that
\begin{equation*}
    \frac{|J_k|}{d_k^3}\lesssim \frac{1}{k^3}\left(\frac{L}{d_k^2}\right)^2.
\end{equation*}
Rearranging the above, it follows that
\begin{equation*}
    d_k\lesssim \frac{L^2}{|J_k|k^3}
\end{equation*}
for an implicit constant independent of $A$. Fixing $A$ to be a constant larger than this implicit one (which itself is absolute), one obtains a contradiction, and therefore concludes that the cellular case does not occur.

Since the cellular case does not occur, the algebraic case holds, hence one may assume without loss of generality that
\begin{equation*}
    J_k\subseteq Z_k.
\end{equation*}

\textbf{Step 3: Reducing to a set $\boldsymbol{\bar{J}_k}$ of joints, with $\boldsymbol{|\bar{J}_k|\sim |J_k|}$, such that
    \begin{itemize}
    \item each $\boldsymbol{x\in \bar{J}_k}$ is a regular point of $\boldsymbol{Z_k}$, and
    \item  each $\boldsymbol{x\in \bar{J}_k}$ lives in a plane $\boldsymbol{\Pi}$ contained in $\boldsymbol{Z_k}$, and $\boldsymbol{\Pi}$ contains $\boldsymbol{\sim k}$ lines in $\boldsymbol{\mathcal{L}}$ through $\boldsymbol{x}$.
    \end{itemize}
 }
Let $\mathcal{P}_k$ be the set of planes inside $Z_k$.\footnote{Non-emptiness of $\mathcal{P}_k$ is not {\em a priori} obvious, however it will follow as a result of our arguments.}

Denote by $\mathcal{L}_{cr,k}$ the set of critical lines in
$\mathcal{L}$, by $\mathcal{L}'_{fl,k}$ the set of flat lines in 
$\mathcal{L}$ that do not lie inside planes of $\mathcal{P}_k$, and by
$\mathcal{L}_{fl,k}$ the set of flat lines in $\mathcal{L}$ that lie inside planes of $\mathcal{P}_k$. The next claim will allow us to assert in the following step that a definite proportion of $J_k$ has planar structure.

\begin{claim}\label{planar k} There exists $\bar{J}_k\subseteq J_k$, with $|\bar{J}_k|\gtrsim |J_k|$, such that each joint in $\bar{J}_k$ is a flat point of $Z_k$, lying in $\sim k$ lines in $\mathcal{L}_{fl,k}$.
\end{claim}

\begin{proof}[Proof of Claim~\ref{planar k}.]  The proof combines a refinement process with a multijoints estimate. 

$\bullet$ We begin with an elementary refinement argument, which shows that, for a definite proportion of the joints in $J_k$, $\sim k$ of the lines in $\mathcal{L}$ through each \textit{lie in $Z_k$}.

Indeed, let $\mathcal{L}'$ be the set of lines in $\mathcal{L}$ each containing $\geq\frac{1}{100} \frac{|J_k|k}{L}$ elements of $J_k$ (i.e., at least the average number of joints). It is easy to see that the lines in $\mathcal{L}'$ are responsible for a definite proportion of the incidences between $J_k$ and $\mathcal{L}$; therefore, 
\begin{equation*}\text{there exist }\gtrsim |J_k|\text{ joints in }J_k\text{ with }\sim k\text{ lines of }\mathcal{L}'\text{ through each.}
\end{equation*}
Denote by $J_k'$ the set of joints with this property.

$\bullet$ Refining once more, we obtain a definite proportion of the joints in $J_k$, such that $\sim k$ of the lines in $\mathcal{L}$ through each are \textit{critical} or \textit{flat} lines in $Z_k$.

Indeed, observe that
\begin{equation*}\frac{|J_k|k}{L}\gtrsim d_k,
\end{equation*}
as otherwise $k$ would be good. In particular, it may be assumed that each line in $\mathcal{L}'$ contains more than $d_k$ elements of $J_k$. This means that \textit{each line in }$\mathcal{L}'$ \textit{lies in }$Z_k$. 

Since $k$ may be assumed to be large enough for at least 3 lines of $\mathcal{L}'$ to pass through each element of $J_k'$, it follows that \textit{each element of $J_k'$ is either a critical or a flat point of $Z_k$}. (The coplanarity condition holds because any three lines contained in $Z_k$ which meet at a regular point must be coplanar.)

The elementary refinement argument is now repeated. More precisely, let $\mathcal{L}''$ be the set of lines in $\mathcal{L}$ each containing $\geq \frac{1}{100}\frac{|J_k'|k}{L}\gtrsim\frac{|J_k|k}{L}$  elements of $J_k'$, for an appropriately small implicit constant. The lines in $\mathcal{L}''$ are responsible for a definite proportion of the incidences between $J_k'$ and $\mathcal{L}$; therefore,
\begin{equation*}
    \text{there exist }\gtrsim |J_k|\text{ joints in }J_k'\text{ each with }\sim k\text{ lines of }\mathcal{L}''\text{ through it.}
\end{equation*}
Importantly, since each element of $J_k'$ is either critical or flat, each line in $\mathcal{L}''$ contains either $>d_k$ critical points of $Z_k$ or $>3d_k-3$ flat points of $Z_k$ (as in the first application of the refinement argument, we may again assume that the quantity $\frac{|J_k|k}{L}$ is larger than an appropriate multiple of $d_k$). Therefore, \textit{each line in }$\mathcal{L}''$ \textit{is either critical or flat}. 

Therefore, for $\gtrsim |J_k|$ joints in $J_k'$, $\sim k$ lines in $\mathcal{L}$ through each are either critical or flat.

$\bullet$ To conclude, we incorporate a multijoints estimate.

The main observation is that since by Propositions~\ref{few critical lines} and \ref{few flat lines} the lines in
$\mathcal{L}_{cr,k}\cup \mathcal{L}'_{fl,k}$ are ``few" (in particular they number $\lesssim d_k^2$), they cannot
be responsible for too many joints in $J_k$. And, therefore, a lot of lines in $\mathcal{L}_{fl,k}$ pass through each one of a definite proportion of the joints in $J_k$.

Indeed, suppose for contradiction that $\gtrsim |J_k|$ joints in
$J_k'$ have the property that each lies in at least 2 lines in
$\mathcal{L}_{cr,k}\cup \mathcal{L}'_{fl,k}$. The joints
with this property are multijoints formed by the three families
$\mathcal{L}_{cr,k}\cup \mathcal{L}'_{fl,k}$, $\mathcal{L}_{cr,k}\cup
\mathcal{L}'_{fl,k}$, $\mathcal{L}$. (Indeed, given two lines in $\mathcal{L}_{cr,k}\cup \mathcal{L}'_{fl,k}$ containing a joint $x$, there must be a third line of $\mathcal{L}$ which is not in the plane formed by these two lines, which together with the first two lines makes $x$ a multijoint.) It follows from the (classical) multijoints theorem discussed in the introduction that
$$|J_k|\lesssim \left( |\mathcal{L}_{cr,k}\cup \mathcal{L}'_{fl,k}|\cdot |\mathcal{L}_{cr,k}\cup \mathcal{L}'_{fl,k}|\cdot L\right)^{1/2}\lesssim (d_k^2\cdot d_k^2\cdot L)^{1/2}\sim d_k^2 L^{1/2},
$$
a contradiction, since $k$ is bad. 

Therefore, $\gtrsim |J_k|$ joints in $J_k'$ have the property that each lies in $\sim k$ lines in $\mathcal{L}_{fl,k}$.

Now, if $\gtrsim |J_k|$ of the above joints were critical, then there would exist $l \in \mathcal{L}_{fl,k}$ containing $\gtrsim \frac{|J_k|k}{L}$ such critical joints. However, since every $\ell \in \mathcal{L}_{fl,k}$ is not a critical line (it is, in fact, a flat line), it contains at most $d_k$ critical points. Therefore, $\frac{|J_k|k}{L}\lesssim d_k$, a contradiction, since $k$ is bad.

It follows that $\gtrsim |J_k|$ joints in $J_k'$ are flat. The set $\bar{J}_k$ of these joints satisfies the statement of the claim.
\end{proof}

\textbf{Step 4: For all bad $\boldsymbol{k}$, $\boldsymbol{\bar{J}_k}$ has planar structure.} This is an easy consequence of the previous step. Indeed, fix any bad $k$. For any plane $\Pi\in \mathcal{P}_k$ (i.e., for any plane lying inside $Z_k$), define
\begin{equation*}
     \bar{J}_{k,\Pi}:=\bar{J}_k\cap \Pi
\end{equation*}
and
\begin{equation*}
    \mathcal{L}^k_\Pi:=\{l\in\mathcal{L}:l\subseteq \Pi\text{ and }l\text{ contains some joint in }\bar{J}_{k,\Pi}\}.
\end{equation*}
 It holds that $\bar{J}_k=\bigcup_{\Pi\in\mathcal{P}_k}\bar{J}_{k,\Pi}$, as the joints in $J_k$ live inside the planes in $\mathcal{P}_k$. The sets $\bar{J}_{k,\Pi}$ are pairwise disjoint, as each joint in $\bar{J}_k$ is a regular point of $Z_k$, and thus cannot live inside two distinct planes in $Z_k$. Finally, the sets $\mathcal{L}^k_\Pi$ are pairwise disjoint as well: if a line $l$ belongs to $\mathcal{L}^k_\Pi$ and $\mathcal{L}^k_{\Pi'}$ for some $\Pi\neq\Pi'$ inside $Z_k$, then $l$ is a critical line of $Z_k$ and therefore cannot contain any regular points of $Z_k$ (contradicting the definitions of $\mathcal{L}^k_\Pi$ and $\mathcal{L}_{\Pi'}^k$).
\begin{remark}\label{qqq}
   {\rm As we have noted,
   in general, a union of sets with planar structure does not have planar structure -- see the discussion in Section~\ref{ppp}. Nevertheless, the study of the interaction between the $Z_k$ corresponding to different bad $k$ reveals that the sets $\bar{J}_k$ above are  exceptional, in that they \textit{have large subsets whose union has planar structure.} The proof of this assertion -- which features below in Step 5 -- builds upon the following (already established) properties of $\bar{J}_k$:
   \begin{enumerate}[(i)]
       \item Each element of $\bar{J}_k$ is a regular point of $Z_k$.
       \item Each $x\in\bar{J}_k$ lies in $\bar{J}_{k,\Pi}$ for the unique $\Pi$ inside $Z_k$ that contains $x$.
       \item The sets $\mathcal{L}^k_\Pi$ (for any fixed bad $k$) are pairwise disjoint.
   \end{enumerate}
   }
\end{remark}

\textbf{Step 5: The set  $\boldsymbol{J_{\text{bad}}=\bigcup_{\text{bad }k}J_k}$ has nearly planar structure.} The following lemma immediately implies that the set
\begin{equation*}
    \bar{J}:=\bigcup_{\text{bad }k}\bar{J}_k,
\end{equation*}
has nearly planar structure. Since $|J_k|\sim |\bar{J}_k|$ for all bad $k$, it directly follows that $J_\text{bad}= \bigcup_{\text{bad }k}J_k$ has nearly planar structure.

\begin{lemma}\label{essence} For each bad $k$, there exists $\bar{J}'_k\subseteq \bar{J}_k$, with $|\bar{J}'_k|\sim |J_k|$, such that $\bigcup_{{\rm bad}\;k}\; \bar{J}'_k$ has planar structure.
\end{lemma}

We break up the proof of Lemma \ref{essence} into three sub-steps.

\textbf{Step 5a: Identification of the enemy.} In this step we partition the joints in $\bar{J}$ into planes in a natural way, and identify the configurations which could obstruct planar (and thus potentially nearly planar) structure for $\bar{J}$ in the context of this partition.

Indeed, we begin by partitioning $\bar{J}$ using the planes inside the collections $\mathcal{P}_k$ (defined in Step 3), over all bad $k$. In particular, recall that, for each bad $k$, all joints in $\bar{J}_k$ are regular points of $Z_k$, lying inside the union of planes $\bigcup_{\Pi\in\mathcal{P}_k}\Pi\subseteq Z_k$. For each bad $k$ and $\Pi\in\mathcal{P}_k$, we have defined $\bar{J}_{k,\Pi}:=\bar{J}_k\cap\Pi$.

As stated in Remark~\ref{qqq}, for any given $k$ the sets $\bar{J}_{k,\Pi}$ are disjoint (and thus form a partition of $\bar{J}_k$).

Let $\mathcal{P}:=\bigcup_{\text{bad }k}\mathcal{P}_k$. For each $\Pi\in\mathcal{P}$, let
\begin{equation*}
    \bar{J}_{\Pi}:=\bigcup_{k:\;  \Pi\in\mathcal{P}_k}\bar{J}_{k,\Pi}.
\end{equation*}
The sets $\bar{J}_\Pi$ are pairwise disjoint (and form a partition of $\bar{J}$). Indeed, for each $x\in \bar{J}$, the $\Pi\in\mathcal{P}$ for which $x\in J_\Pi$ is the unique plane inside $Z_k$ that contains $x$, for the unique $k$ for which $x\in J_k$. Define
\begin{equation*}
    \mathcal{L}_{\Pi}:=\{l\in\mathcal{L}:l\subseteq \Pi\text{ and }l\text{ contains some point in }\bar{J}_\Pi\}
\end{equation*}
and observe that for any $\Pi\in\mathcal{P}$ it holds that $\mathcal{L}_{\Pi}=\bigcup_{\text{bad }k}\mathcal{L}_\Pi^k$, where, recall,
\begin{equation*}
      \mathcal{L}_{\Pi}^k:=\{l\in\mathcal{L}:l\subseteq \Pi \text{ and }l\text{ contains some point in }\bar{J}_{k,\Pi}\}.
 \end{equation*}
 If the sets $\mathcal{L}_\Pi$ are pairwise disjoint, then $\bar{J}$ has planar structure. In order to study the interaction of the sets $\mathcal{L}_{\Pi}$, for any bad $k$ define
\begin{equation*}
      \mathcal{L}^k:=\bigsqcup_{\Pi\in\mathcal{P}_k}\mathcal{L}_{\Pi}^k.
 \end{equation*}
\begin{figure}[htbp]
\centering

\tdplotsetmaincoords{70}{110}
\begin{tikzpicture}[tdplot_main_coords,font=\sffamily, scale=0.74];

\draw[fill=red,opacity=0.2] (0,0,3) -- (-5,-5,-2) -- (-5,-5,-8) -- (0,0,-3);
\draw[black, thin] (0,0,3) -- (-5,-5,-2) -- (-5,-5,-8) -- (0,0,-3);

\node[anchor=south west,align=center] (line) at (-10,-11,-4) {\large{\color{red}$\Pi=\Pi_\ell^k\subseteq Z_k$}};
;

\draw[fill=blue,opacity=0.2] (0,0,-3) -- (0,0,3) -- (1,3.5,1) -- (1,3.5,-5) -- cycle;
\draw[black, thin] (0,0,-3) -- (0,0,3) -- (1,3.5,1) -- (1,3.5,-5) -- cycle;

\node[anchor=south west,align=center] (line) at (1,3.8,1) {\large{\color{blue}$\Pi'=\Pi_\ell^{k'}\subseteq Z_{k'}$}};
;

\draw[color=red!60!blue, ultra thick] (0,0,-3) -- (0,0,3);

\node[anchor=south west,align=center] (line) at (-1,-1,2.7) {\large{\color{red!50!blue}$\ell=\Pi\cap \Pi'$}};
;

\fill[red] (0,0,1.5) circle (1mm);
\draw[red, thick] (0,0,1.5) -- (-1.3,-1.3,1);
\draw[red, thick] (0,0,1.5) -- (-3.5,-3.5,-2);
\draw[red, thick] (0,0,1.5) -- (-2.2,-2.2,-1.5);

\fill[blue] (0,0,0.2) circle (1mm);
\draw[blue, thick] (0,0,0.2) -- (0.2, 0.7,1);
\draw[blue, thick] (0,0,0.2) -- (0.4, 1.4,0.8);
\draw[blue, thick] (0,0,0.2) -- (0.6, 2.1,0);
\draw[blue, thick] (0,0,0.2) -- (0.4, 1.4,-0.5);

\fill[red] (0,0,-1) circle (1mm);
\begin{scope}[shift={(0,0,-2.5)}]
\draw[red, thick] (0,0,1.5) -- (-1.3,-1.3,1);
\draw[red, thick] (0,0,1.5) -- (-3.5,-3.5,-2);
\draw[red, thick] (0,0,1.5) -- (-2.2,-2.2,-1.5);
\end{scope}

\fill[blue] (0,0,-2) circle (1mm);
\begin{scope}[shift={(0,0,-2.2)}]
\draw[blue, thick] (0,0,0.2) -- (0.2, 0.7,1);
\draw[blue, thick] (0,0,0.2) -- (0.4, 1.4,0.8);
\draw[blue, thick] (0,0,0.2) -- (0.6, 2.1,0);
\draw[blue, thick] (0,0,0.2) -- (0.4, 1.4,-0.5);
\end{scope}

\node[anchor=south west,align=center] (line) at (-3.2,-3.2,-0.7) {\color{red} $\sim k$};
;
\node[anchor=south west,align=center] (line) at (-3.2,-3.2,-3.2) {\color{red} $\sim k$};
;

\node[anchor=south west,align=center] (line) at (0.4,1.4,0) {\color{blue} $\sim k'$};
;
\node[anchor=south west,align=center] (line) at (0.4,1.4,-2.2) {\color{blue} $\sim k'$};
;

\end{tikzpicture}

 \captionsetup{singlelinecheck=off}
\caption[.]{\small{We demonstrate the situation which could obstruct planar structure (and could thus potentially also obstruct nearly planar structure) for $\bar{J}$. The $k$, $k'$ are both bad. The line $\ell$ is the intersection of two distinct planes, one in $Z_k$ (red) and one in $Z_{k'}$ (blue), and carries simultaneously joints in $\bar{J}_k$ (red) and in $\bar{J}_{k'}$ (blue). The bulk of the lines through each of the red joints are in $\mathcal{L}^k$, and they are flat lines of $Z_k$ inside $\Pi$. The bulk of the lines through each of the blue joints are in $\mathcal{L}^{k'}$, and they are flat lines of $Z_{k'}$ inside $\Pi'$. The line $\ell$ itself lies in $\mathcal{L}^k\cap\mathcal{L}^{k'}$.}}
\label{fig: obstruction to planar structure}
\end{figure}
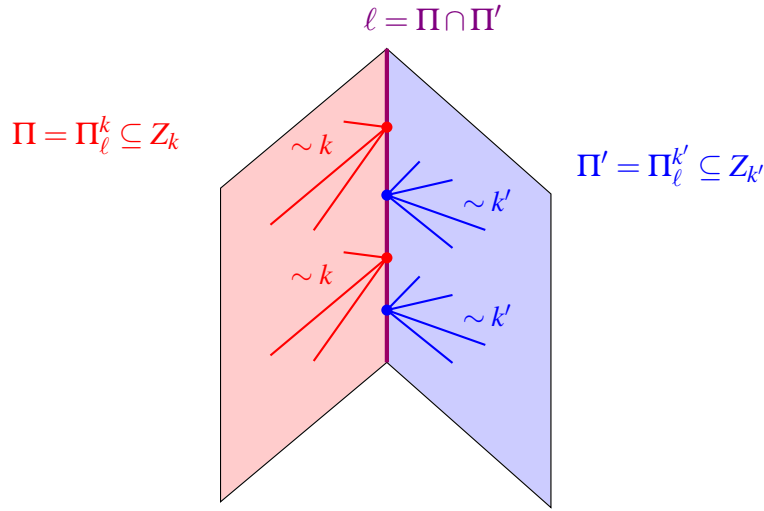  
 For any $\ell\in\mathcal{L}^k$, denote by $\Pi_{\ell}^k$ the unique $\Pi\in\mathcal{P}_k$ for which $\ell\in\mathcal{L}_\Pi^k$ (the unique $\Pi\in\mathcal{P}_k$ in which $\ell$ lies). Note that this $\Pi$ is unique, due to the disjointness of the sets $\mathcal{L}^k_\Pi$ for fixed $k$ (see Remark~\ref{qqq}).

Suppose that the sets $\mathcal{L}_\Pi$ are not pairwise disjoint. This means that there exists a line $\ell\in \mathcal{L}_\Pi\cap\mathcal{L}_{\Pi'}$ for some $\Pi\neq\Pi'$ in $\mathcal{P}$; in particular, there exist bad $k,k'$ so that $\ell\in\mathcal{L}_\Pi^k\cap \mathcal{L}_{\Pi'}^{k'}$. This implies that $k\neq k'$ (as the sets $\mathcal{L}_{\Pi}^k,\mathcal{L}_{\Pi'}^k$ are disjoint). Therefore, $\ell\in\mathcal{L}^k\cap\mathcal{L}^{k'}$ for these distinct $k,k'$, and moreover the planes $\Pi_{\ell}^k=\Pi$ and $\Pi_{\ell}^{k'}=\Pi'$ are distinct. This situation is depicted in Figure~\ref{fig: obstruction to planar structure}.

We have thus demonstrated that the only potential obstruction to $\bar{J}$ having planar structure in the context of our partition would be the existence of some line $\ell$ that lives simultaneously in two sets $\mathcal{L}^k$, $\mathcal{L}^{k'}$ for $k\neq k'$, and additionally satisfies $\Pi_{\ell}^k\neq\Pi_{\ell}^{k'}$. We are not disproving the existence of such a problematic line here. However, the technical Claim~\ref{interaction} in Step 5b below will imply (in Step 5c) that, even if such problematic lines exist (causing potential obstructions to planar structure), they still cannot obstruct nearly planar structure. In particular, \textit{the lines through each joint in a large subset of $\bar{J}=\bigcup_{\text{bad }k}\; \bar{J}_k$ are not problematic.} The algebraic-geometric nature of obstructions to planar structure (see Figure~\ref{fig: obstruction to planar structure}) prompts us to show this by exploring how different varieties $Z_k$, $Z_{k'}$ (for bad $k$, $k'$) interact with each other.

{ \bf Step 5b: Interaction of the varieties $\boldsymbol{Z_k}$ as bad $\boldsymbol{k}$ varies.} To study this interaction in a manner which is helpful for the proof of Claim~\ref{interaction}, we define a total
order $\prec$ on the set of bad $k$ such that
\begin{equation*}
    \text{if }k'\prec k\text{, then }k'^{\epsilon}d_{k'}\leq k^{\epsilon}d_k.
\end{equation*}
This is achieved by simply ordering the quantities $k^{\epsilon}d_k$ in (usual) increasing order, and assigning the same order to the corresponding $k$'s. (For $k$'s for which the corresponding $k^{\epsilon}d_k$ are equal, any total order between them is permitted.)

To formulate the claim, for any $x\in \bar{J}$ define
\begin{equation*}
    \mathcal{L}_x:=\{l\in\mathcal{L}_\Pi\text{ through }x\}
\end{equation*}
for the unique $\Pi\in\mathcal{P}$ for which $x\in J_\Pi$. Observe that if $x\in \bar{J}_k$ then all lines in $\mathcal{L}_x$ belong to $\mathcal{L}^k$ (and total $\sim k$ in number).

\begin{claim}\label{interaction} For all bad $k$, there exists $\bar{J}'_k\subseteq \bar{J}_k$, with $|\bar{J}'_k|\sim |J_k|$, such that any line $\ell\in \bigcup_{x\in \bar{J}_k'}\mathcal{L}_x$ with
\begin{equation*}
    \ell \in \mathcal{L}^k \cap \bigcup_{k'\prec k}\;\mathcal{L}^{k'}
\end{equation*}
satisfies
\begin{equation*}
  \Pi^{k'}_\ell = \Pi^k_\ell.  
\end{equation*}
\end{claim}

Roughly speaking, Claim~\ref{interaction} states that, for any bad $k$, only a small proportion of the joints in $\bar{J}_k$ may live in lines $\ell$ as in the obstructive Figure~\ref{fig: obstruction to planar structure}, for $k'\prec k$.

\begin{proof}[Proof of Claim~\ref{interaction}.] Fix a bad
  $k$. For $x\in \bar{J}_k$, we say that $x$ is \textit{problematic} if there exists $\ell\in \mathcal{L}_x$ with
\begin{equation*}
    \ell\in \bigcup_{k'\prec k}\;\mathcal{L}^{k'}\text{ and }\;\Pi^{k'}_\ell\neq \Pi^k_\ell.
\end{equation*}
More precisely, for $k'\prec k$, we say that $x$ is $k'$\textit{-problematic} if there exists $\ell\in \mathcal{L}_x$ with
\begin{equation*}
    \ell\in\mathcal{L}^{k'}\text{ and }\Pi^{k'}_\ell\neq \Pi^k_\ell.
\end{equation*}
In other words, $x\in\bar{J}_k$ is $k'$-problematic \textit{if it is a red joint inside some line $\ell$ as in Figure~ \ref{fig: obstruction to planar structure}.}

Denote by $\bar{J}_{k,\text{prob}}$ and  $\bar{J}^{k'}_{k,\text{prob}}$ the sets of problematic and $k'$-problematic joints, respectively. Observe that
\begin{equation}\label{eq:reduce to single k'}\bar{J}_{k,\text{prob}}=\bigcup_{\text{bad }k'\prec k}\;\;\bar{J}_{k,\text{prob}}^{k'}.
\end{equation}
The goal is to prove that 
\begin{equation*}
    |\bar{J}_{k,\text{prob}}|\lesssim |J_k|
\end{equation*}
(for an appropriate implicit constant smaller than 1); that is, that only a small proportion of the joints in $\bar{J}_k$ may live in lines $\ell$ as in Figure~\ref{fig: obstruction to planar structure}, for $k'\prec k$.

Suppose for contradiction that $|\bar{J}_{k,\text{prob}}|\gtrsim |J_k|$. It follows that there exists a bad $k'\prec k$ such that
\begin{equation*}
    |\bar{J}_{k,\text{prob}}^{k'}|\gtrsim_{\epsilon} \frac{|J_k|}{k'^{\epsilon}}
\end{equation*}
(since otherwise \eqref{eq:reduce to single k'} would imply that $|\bar{J}_{k,\text{prob}}| \leq \sum_{\text{bad }k'\prec k} |\bar{J}_{k,\text{prob}}^{k'}| \lesssim_\epsilon \sum_{\text{bad }k'\prec k} 
\frac{|J_k|}{k'^{\epsilon}} \lesssim |J_k|$, a contradiction under the assumption that the $\epsilon$-dependent implicit constant above is appropriately small).

We will derive a contradiction by appropriately bounding $I(\bar{J}_{k,\text{prob}}^{k'},\mathcal{L}^k)$, the number of incidences between the joints in $\bar{J}_{k,\text{prob}}^{k'}$ and the lines in $\mathcal{L}^k$. The core of the analysis is informally summarised in Figure~ \ref{fig: interaction} below, which builds on Figure~\ref{fig: obstruction to planar structure}.

To begin with, the planar structure of $\bar{J}_k$ (more precisely, the fact that $\sim k$ lines in $\mathcal{L}^k$ pass through each joint in $\bar{J}_k$) and the assumed lower bound on $|\bar{J}_{k,\text{prob}}^{k'}|$ imply that
\begin{equation*}
    I(\bar{J}_{k,\text{prob}}^{k'}, \mathcal{L}^k)\sim |\bar{J}_{k,\text{prob}}^{k'}|\cdot k\gtrsim_\epsilon \frac{|J_k|}{k'^{\epsilon}}\cdot k.
\end{equation*}
On the other hand, 
\begin{equation}\label{splitting the incidences}
    I(\bar{J}_{k,\text{prob}}^{k'}, \mathcal{L}^k)=I(\bar{J}_{k,\text{prob}}^{k'}, \mathcal{L}^k_{\subsetneq Z_{k'}})+I(\bar{J}_{k,\text{prob}}^{k'}, \mathcal{L}^k_{\subseteq Z_{k'}}),
\end{equation}
where $\mathcal{L}^k_{\subsetneq Z_{k'}}$ is the set of lines in $\mathcal{L}^k$ that do not lie fully inside $Z_{k'}$, and where $\mathcal{L}^k_{\subseteq Z_{k'}}$ is the set of lines in $\mathcal{L}^k$ that lie fully inside $Z_{k'}$. We split the analysis into two cases, according to which of the summands is dominant in \eqref{splitting the incidences}. Recall, each $x\in\bar{J}_{k,\text{prob}}^{k'}$ lies in $Z_k\cap Z_{k'}$.

$\bullet$ \textit{Suppose that }$I(\bar{J}_{k,\text{prob}}^{k'}, \mathcal{L}^k)\sim I(\bar{J}_{k,\text{prob}}^{k'}, \mathcal{L}^k_{\subsetneq Z_{k'}})$. It directly follows that
\begin{equation*}
    I( \bar{J}_{k,\text{prob}}^{k'}, \mathcal{L}^k_{\subsetneq Z_{k'}})\gtrsim_\epsilon \frac{|J_k|}{k'^{\epsilon}}\cdot k.
\end{equation*}
On the other hand, since $ \bar{J}_{k,\text{prob}}^{k'}\subseteq Z_{k'}$, each line in $\mathcal{L}^k_{\subsetneq Z_{k'}}$ contains at most $d_{k'}$ elements of $ \bar{J}_{k,\text{prob}}^{k'}\subseteq Z_{k'}$. Therefore,
\begin{equation*}
    I( \bar{J}_{k,\text{prob}}^{k'}, \mathcal{L}^k_{\subsetneq Z_{k'}})\leq L d_{k'}.\end{equation*}
It follows by the two estimates above that
\begin{equation*}
    \frac{|J_k|}{k'^{\epsilon}}\cdot k\lesssim_\epsilon Ld_{k'}.
\end{equation*}
This situation is illustrated by (ii) in Figure~\ref{fig: interaction}.

$\bullet$ \textit{Suppose that} $I(\bar{J}_{k,\text{prob}}^{k'}, \mathcal{L}^k)\sim I(\bar{J}_{k,\text{prob}}^{k'}, \mathcal{L}^k_{\subseteq Z_{k'}})$. It directly follows that
\begin{equation*}
    I( \bar{J}_{k,\text{prob}}^{k'}, \mathcal{L}^k_{\subseteq Z_{k'}})\gtrsim_\epsilon \frac{|J_k|}{k'^{\epsilon}}\cdot k.
\end{equation*}

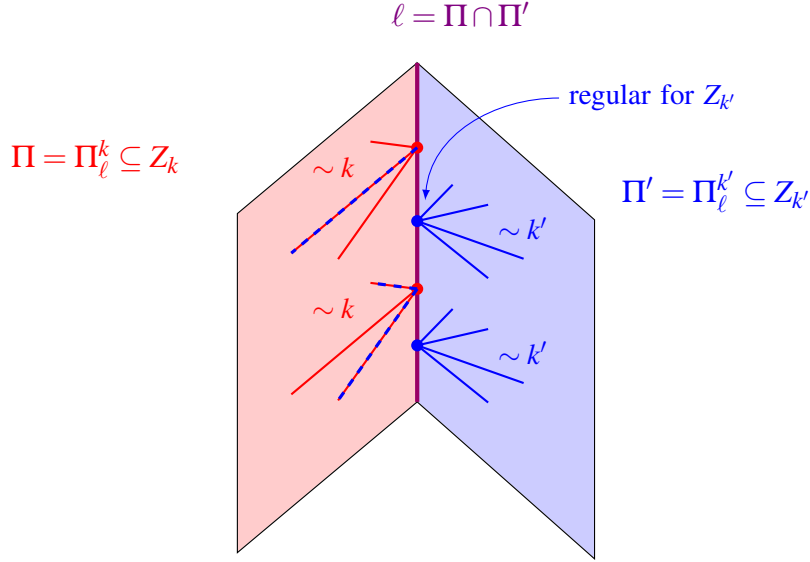
\begin{figure}[htbp]
\centering

\tdplotsetmaincoords{70}{110}
\begin{tikzpicture}[tdplot_main_coords,font=\sffamily, scale=0.8];

\draw[fill=red,opacity=0.2] (0,0,3) -- (-5,-5,-2) -- (-5,-5,-8) -- (0,0,-3);
\draw[black, thin] (0,0,3) -- (-5,-5,-2) -- (-5,-5,-8) -- (0,0,-3);

\node[anchor=south west,align=center] (line) at (-10,-11,-4) {\large{\color{red}$\Pi=\Pi_\ell^k\subseteq Z_k$}};
;

\draw[fill=blue,opacity=0.2] (0,0,-3) -- (0,0,3) -- (1,3.5,1) -- (1,3.5,-5) -- cycle;
\draw[black, thin] (0,0,-3) -- (0,0,3) -- (1,3.5,1) -- (1,3.5,-5) -- cycle;

\node[anchor=south west,align=center] (line) at (1,3.8,1) {\large{\color{blue}$\Pi'=\Pi_\ell^{k'}\subseteq Z_{k'}$}};
;

\draw[color=red!60!blue, ultra thick] (0,0,-3) -- (0,0,3);

\node[anchor=south west,align=center] (line) at (-1,-1,3) {\large{\color{red!50!blue}$\ell=\Pi\cap \Pi'$}};
;

\fill[red] (0,0,1.5) circle (1mm);
\draw[red, thick] (0,0,1.5) -- (-1.3,-1.3,1);
\draw[red, thick] (0,0,1.5) -- (-3.5,-3.5,-2);
\draw[red, thick] (0,0,1.5) -- (-2.2,-2.2,-1.5);
\draw[blue, very thick, dashed] (0,0,1.5) -- (-3.5,-3.5,-2);

\fill[blue] (0,0,0.2) circle (1mm);
\draw[blue, thick] (0,0,0.2) -- (0.2, 0.7,1);
\draw[blue, thick] (0,0,0.2) -- (0.4, 1.4,0.8);
\draw[blue, thick] (0,0,0.2) -- (0.6, 2.1,0);
\draw[blue, thick] (0,0,0.2) -- (0.4, 1.4,-0.5);
\draw[-latex, blue]
     (0.8, 2.8,3) [out=180, in=85] to (0.3,0.25,0.7) ;
\node[anchor=south west,align=center] (line) at (0.8, 2.8,2.6) {\color{blue}${\rm regular}$ ${\rm for}$ $Z_{k'}$};
;

\fill[red] (0,0,-1) circle (1mm);
\begin{scope}[shift={(0,0,-2.5)}]
\draw[red, thick] (0,0,1.5) -- (-1.3,-1.3,1);
\draw[red, thick] (0,0,1.5) -- (-3.5,-3.5,-2);
\draw[red, thick] (0,0,1.5) -- (-2.2,-2.2,-1.5);
\draw[blue, very thick, dashed] (0,0,1.5) -- (-1.3,-1.3,1);
\draw[blue, very thick, dashed] (0,0,1.5) -- (-2.2,-2.2,-1.5);
\end{scope}

\fill[blue] (0,0,-2) circle (1mm);
\begin{scope}[shift={(0,0,-2.2)}]
\draw[blue, thick] (0,0,0.2) -- (0.2, 0.7,1);
\draw[blue, thick] (0,0,0.2) -- (0.4, 1.4,0.8);
\draw[blue, thick] (0,0,0.2) -- (0.6, 2.1,0);
\draw[blue, thick] (0,0,0.2) -- (0.4, 1.4,-0.5);
\end{scope}

\node[anchor=south west,align=center] (line) at (-3.2,-3.2,-0.7) {\color{red} $\sim k$};
;
\node[anchor=south west,align=center] (line) at (-3.2,-3.2,-3.2) {\color{red} $\sim k$};
;

\node[anchor=south west,align=center] (line) at (0.4,1.4,0) {\color{blue} $\sim k'$};
;
\node[anchor=south west,align=center] (line) at (0.4,1.4,-2.2) {\color{blue} $\sim k'$};
;

\end{tikzpicture}

 \captionsetup{singlelinecheck=off}
\caption[.]{\small{The problematic joints (red) are arranged in lines $\ell\in \mathcal{L}^k\cap\mathcal{L}^{k'}$ as above (a configuration precisely as in Figure~\ref{fig: obstruction to planar structure}). Everything that is red lives in $Z_k$, and everything that is blue lives in $Z_{k'}$ (and $\ell$ itself lives in both zero sets). The lines that are only red do not live in $Z_{k'}$, while the red lines that are also blue-dashed live in $Z_{k'}$. The red joints are regular points of $Z_k$, while the blue joints are regular points of $Z_{k'}$. The problematic joints are few: (i) The existence of regular blue points along $\ell$ means that the red plane $\Pi$ \textit{cannot also be blue} (i.e., cannot live in $Z_{k'}$), so \textit{the blue-dashed lines are few}, contributing few incidences with the red joints. (ii) The red lines contribute few incidences with the red joints as each can only intersect the blue zero set few times.}}
\label{fig: interaction}
\end{figure} 
More precisely, one obtains that 
\begin{equation}\label{further property flat case}\gtrsim |\bar{J}_{k,\text{prob}}^{k'}|\gtrsim_\epsilon \frac{|J_k|}{k'^{\epsilon}}\text{ joints }x\in \bar{J}_{k,{\rm prob}}^{k'}\text{ have }\sim k\text{ lines in }\mathcal{L}^k_{\subseteq Z_{k'}}\text{ through each.}
\end{equation}
Since the joints in $\bar{J}_{k,\text{prob}}^{k'}$ lie in the lines in the set
\begin{equation*}
    \{\ell\in\mathcal{L}^k\cap\mathcal{L}^{k'}:\;\Pi^k_\ell\neq\Pi^{k'}_\ell\}\subseteq \mathcal{L},
\end{equation*}
it follows that there exists $\ell\in\mathcal{L}^k\cap\mathcal{L}^{k'}$ with $\Pi^k_\ell\neq\Pi^{k'}_\ell$ which contains $\gtrsim_\epsilon \frac{|J_k|}{k'^{\epsilon}L}$ of the joints in \eqref{further property flat case}. The $\sim k$ lines in $\mathcal{L}^k$ through each of these joints all lie in $\Pi^k_{\ell}$ (as these joints are regular points of $Z_k$ on the plane $\Pi_\ell^k\subseteq Z_k$, and the lines in $\mathcal{L}^k$ lie in $Z_k$); thus, by \eqref{further property flat case}, $\Pi^k_{\ell}$ contains $\gtrsim_\epsilon \frac{|J_k|}{k'^{\epsilon}L}k$ lines that all lie in $Z_{k'}$. However, these lines are fewer than $d_{k'}$ in total, as otherwise $\Pi^k_{\ell}$ would lie in $Z_{k'}$, and thus $\ell$, a line that contains at least one regular point of $Z_{k'}$, would be the intersection of two distinct planes in $Z_{k'}$, which cannot happen. Therefore,
\begin{equation*}
    \frac{|J_k|}{k'^{\epsilon}L}k\lesssim_\epsilon d_{k'}.
\end{equation*}
This situation is illustrated by (i) in Figure~\ref{fig: interaction}.
\bigskip

Observe that both cases above lead to the same bound
\begin{equation*}
    |J_k|k \lesssim_\epsilon Lk'^{\epsilon}d_{k'}
\end{equation*}
for $J_k$. Now, due to the fact that $k'\prec k$ it holds that $k'^{\epsilon}d_{k'}\leq k^{\epsilon}d_{k}$, thus
\begin{equation*}
   |J_k|k \lesssim_\epsilon Lk^{\epsilon}d_{k},
\end{equation*}
or equivalently $|J_k|k^{2-\epsilon/2}\lesssim_\epsilon L^{3/2}$,
which is a contradiction because $k$ is bad.

Therefore, $|\bar{J}_{k,\text{prob}}|\lesssim |J_k|$; the proof of Claim~\ref{interaction} is complete.

\end{proof}

{\bf Step 5c: Conclusion: the failure of the enemy.}
Claim~\ref{interaction} implies that $\bar{J}':=\bigcup_{\text{bad }k}\bar{J}'_k$ has planar structure. Indeed, for all $k$ and $\Pi\in\mathcal{P}=\bigcup_{\text{bad }k}\mathcal{P}_k$ define
\begin{equation*}
    \bar{J}'_{k,\Pi}:=\bar{J}_k'\cap \bar{J}_{k,\Pi}
\end{equation*}
and
\begin{equation*}
    \bar{J}'_\Pi:=\bigsqcup_{k:\Pi\in\mathcal{P}_k}\bar{J}'_{k,\Pi}.
\end{equation*}
The sets $\bar{J}'_{\Pi}$ are pairwise disjoint, as each $x\in \bar{J}$ (and thus in $\bar{J}'$) belongs to $\bar{J}'_{\Pi}$ for the unique $\Pi\in\mathcal{P}_k$ that contains $x$, for the unique $k$ for which $x\in\bar{J}'_k$. Therefore, to show that $\bar{J}'$ has planar structure it suffices to show that the sets
\begin{equation*}
    \bar{\mathcal{L}}_\Pi:=\{l\in\mathcal{L}:l\subseteq \Pi\text{ and }l\text{ contains some joint in }\bar{J}'_{\Pi}\}
\end{equation*}
are pairwise disjoint.

Assume for contradiction that the sets $\bar{\mathcal{L}}_\Pi$ are not pairwise disjoint. This means that there exists a line $\ell\in\bar{\mathcal{L}}_\Pi\cap\bar{\mathcal{L}}_{\Pi'}$ for some $\Pi\neq\Pi'$ in $\mathcal{P}$. Since $\ell\in\mathcal{L}_\Pi$, it follows that $\ell$ is contained in $\Pi$ and contains some joint $x\in\bar{J}'_{k,\Pi}$, for some $k$. This further implies that  $\ell\in \bigcup_{x\in \bar{J}'_k}\mathcal{L}_x$, $\ell\in \mathcal{L}^k$ and $\Pi_\ell^k=\Pi\in\mathcal{P}_k$. Similarly, the fact that $\ell\in\bar{\mathcal{L}}_{\Pi'}$ implies that $\ell$ contains some joint in $\bar{J}'_{k',\Pi}$ for some $k'$, and therefore that $\ell\in \bigcup_{x\in \bar{J}'_{k'}}\mathcal{L}_x$, $\ell\in\mathcal{L}^{k'}$ and $\Pi_{\ell}^{k'}=\Pi'\in\mathcal{P}_{k'}$. 

It is impossible for the above to hold for $k=k'$. Indeed, if this was the case, then $\ell$ would be the intersection of the two distinct planes $\Pi,\Pi'$, which both lie in $Z_k$. Thus all points in $\ell$ would be critical points of $Z_k$, and therefore $\ell$ would not contain any element of $\bar{J}'_k$, a contradiction.

It follows that $k\neq k'$. It has thus been shown that for these distinct $k,k'$
\begin{equation*}
    \ell\in \bigcup_{x\in\bar{J}'_k}\mathcal{L}_x\text{ and }\ell\in\bigcup_{x\in\bar{J}'_{k'}}\mathcal{L}_x
\end{equation*}
while also
\begin{equation*}
    \ell\in\mathcal{L}^k\cap\mathcal{L}^{k'}.
\end{equation*}
Now, either $k'\prec k$ or $k\prec k'$. If $k'\prec k$, then the above implies that $\ell\in\bigcup_{x\in\bar{J}'_k}\mathcal{L}_x$ and $\ell\in \mathcal{L}^k\cap\bigcup_{k'\prec k}\mathcal{L}^{k'}$; by Claim~ \ref{interaction} it follows that $\Pi=\Pi'$, a contradiction. Similarly, if $k\prec k'$ the above implies that $\ell\in \bigcup_{x\in\bar{J}'_{k'}}\mathcal{L}_x$ and $\ell\in \mathcal{L}^{k'}\cap \bigcup_{k\prec k'}\mathcal{L}^{k'}$, which again leads to the contradiction $\Pi=\Pi'$ by Claim~\ref{interaction}.

Therefore, the sets $\bar{\mathcal{L}}_\Pi$ are pairwise disjoint. It follows that $\bar{J}'$ has planar structure -- the proof of Lemma~\ref{essence} is complete.


\textbf{Step 6: Proving the discrete Kakeya estimate.} To complete the proof of Theorem~\ref{3d}, it remains to show the discrete Kakeya-type estimate \eqref{eq:discrete kakeya}. We begin by showing that 
it holds under the additional hypothesis of nearly planar structure, with the aid of the Szemer\'edi--Trotter theorem. 

In crude terms, each plane $\Pi$ featuring in a nearly planar structure is \textit{independent} from the other planes, when it comes to counting incidences. In particular, it is the lines from \textit{within} $\Pi$ that contribute essentially all incidences with the points that have chosen $\Pi$. This fact informs the basic idea for the proof of Lemma~\ref{lemma:nearlyplanarest}: finding an appropriate incidence estimate on each such plane $\Pi$, and then summing over all $\Pi$.

\begin{lemma}\label{lemma:nearlyplanarest}
Let $J$ be a set of joints formed by a set $\mathcal{L}$ of $L$ lines in $\mathbb{R}^3$. If $J$ has nearly planar structure, then
\begin{equation*}
    \sum_{x\in J}\left(\sum_{l\in\mathcal{L}}\chi_l(x)\right)^{3/2}\lesssim L^{3/2}.
\end{equation*}
\end{lemma}

\begin{proof}
The lemma is proved for sets of joints with planar structure; the general statement immediately follows by the definition of nearly planar structure.

Let $J$ be a set of joints formed by $\mathcal{L}$ that has planar structure. As before, the desired inequality becomes
\begin{equation*}
    \sum_{k}|J_k|k^{3/2}\lesssim L^{3/2}.
\end{equation*}
Since $J$ has planar structure, there exist a set $\mathcal{P}$ of planes and a decomposition $J=\bigsqcup_{\Pi\in\mathcal{P}}J_{\Pi}$ in sets $J_{\Pi}\subseteq J\cap \Pi$, so that the sets
\begin{equation*}
    \mathcal{L}_\Pi:=\{l\in\mathcal{L}:l\subseteq \Pi\text{ and }l\text{ contains some point in } J_\Pi \},
\end{equation*}
whose cardinalities we denote by $L_\Pi$, are pairwise disjoint and satisfy \begin{equation*}
     \#\{\text{lines in }\mathcal{L}_{\Pi}\text{ through }x\}\sim \#\{\text{lines in }\mathcal{L}\text{ through }x\}
\end{equation*}
for every joint $x\in J_\Pi$. Observe that $J_k=\bigsqcup_{\Pi\in\mathcal{P}}J_{k,\Pi}$, where $J_{k,\Pi}$ is the set of joints in $J_\Pi\cap J_k$. The desired inequality thus becomes
\begin{equation}\label{eq:structural estimate}
      \sum_{\Pi\in\mathcal{P}}\;\sum_{k}|J_{k,\Pi}|k^{3/2}\lesssim L^{3/2},
\end{equation}
and will follow from the ``pointwise" estimate
\begin{equation}\label{eq:planar structural estimate}
    \sum_{k}|J_{k,\Pi}|k^{3/2}\lesssim L_\Pi L^{1/2}\text{ for all }\Pi\in\mathcal{P}
\end{equation}
by adding over all $\Pi\in\mathcal{P}$, crucially using the fact that
\begin{equation*}
    \sum_{\Pi\in\mathcal{P}}L_\Pi\leq L,
\end{equation*}
which holds because the sets $\mathcal{L}_\Pi$ are pairwise disjoint.

We now show \eqref{eq:planar structural estimate} to complete the proof. Let $\Pi\in\mathcal{P}$. Observe that, due to the planar structure of $J$, each joint in $J_{k,\Pi}$ lies in $\sim k$ lines in $\mathcal{L}_\Pi$. Therefore, the desired estimate
\begin{equation}\label{eq:fixed plane estimate}
    \sum_{k}|J_{k,\Pi}|k^{3/2}\lesssim L_\Pi L^{1/2}
\end{equation}
is a statement regarding incidences between $J_{k,\Pi}$ and $\mathcal{L}_\Pi$, and the Szemer\'edi--Trotter theorem will be employed for its proof. In particular, for $k\gtrsim L_\Pi^{1/2}$, applying the Szemer\'edi--Trotter theorem to count incidences between $J_{k,\Pi}$ and $\mathcal{L}_\Pi$, one obtains
\begin{equation*}
    |J_{k,\Pi}|\lesssim \frac{L_\Pi}{k}.
\end{equation*}
Therefore,
\begin{equation*}
    \sum_{k\gtrsim L_{\Pi}^{1/2}}|J_{k,\Pi}|k^{3/2}\lesssim L_{\Pi}\sum_{k\gtrsim L_{\Pi}^{1/2}}k^{1/2}\lesssim L_\Pi^{3/2}\lesssim L_\Pi L^{1/2}.
\end{equation*}
On the other hand, the Szemer\'edi--Trotter theorem asserts that for $k\lesssim L_\Pi^{1/2}$ the inequality
\begin{equation}\label{fact1}
    |J_{k,\Pi}|\lesssim \frac{L_\Pi^2}{k^3}
\end{equation}
holds. Moreover, the joints structure may be exploited to derive
\begin{equation} \label{fact2}
    |J_{\Pi}|\leq L.
\end{equation}
Indeed, all points in $J_{\Pi}$ lie on the same plane $\Pi$, however they are joints formed by $\mathcal{L}$; hence, there exists a distinct line in $\mathcal{L}$ through each joint in $J_\Pi$ (which does not lie in $\Pi$), and therefore $L\geq |J_\Pi|$. Inequalities \eqref{fact1} and \eqref{fact2} will now be combined to derive the estimate
\begin{equation}\label{interesting k on plane}
    \sum_{k\lesssim L_{\Pi}^{1/2}}|J_{k,\Pi}|k^{3/2}\lesssim L_\Pi L^{1/2},
\end{equation}
concluding the proof. The analysis is split in two cases, according to whether $k\lesssim Q$ or $k\gtrsim Q$, where 
\begin{equation*}
    Q:=\left(\frac{L_{\Pi}^2}{L}\right)^{1/3}.
\end{equation*}
The former case is resolved by exploiting the joints structure (in particular, \eqref{fact2}). More precisely,
\begin{equation}\label{improved S-T}
    \sum_{k\lesssim Q}|J_{k,\Pi}|k^3\lesssim \left(\sum_{k\lesssim Q}|J_{k,\Pi}|\right) Q^3\lesssim |J_{\Pi}|\frac{L_{\Pi}^2}{L}\lesssim L\frac{L_{\Pi}^2}{L}= L_{\Pi}^2,
\end{equation}
where the last inequality is \eqref{fact2}. (Note that the above estimate may be viewed as an improved version of the Szemer\'edi--Trotter theorem for $\bigsqcup_{k\lesssim Q}J_{k,\Pi}$, as merely applying \eqref{fact1} for each $k\lesssim Q$ and adding over all such $k$ will in general yield the above inequality with a $\log Q$ loss.) Applying the Cauchy-Schwarz inequality and using again the joints structure estimate \eqref{fact2}, this time combined with \eqref{improved S-T}, one deduces
\begin{eqnarray} \label{small k on plane}
   \begin{aligned}
       \sum_{k\lesssim Q}|J_{k,\Pi}|k^{3/2}&= \sum_{k\lesssim Q}(|J_{k,\Pi}|k^3)^{1/2} |J_{k,\Pi}|^{1/2}\\
       &\lesssim \left(\sum_{k\lesssim Q}|J_{k,\Pi}|k^3\right)^{1/2}\left( \sum_{k\lesssim Q}|J_{k,\Pi}|\right)^{1/2}\\
       &\lesssim\left(\sum_{k\lesssim Q}|J_{k,\Pi}|k^3\right)^{1/2} |J_{\Pi}|^{1/2}\\
       &\lesssim L_{\Pi} L^{1/2}.
    \end{aligned}
\end{eqnarray}
 The situation for the remaining $k$ (those for which $Q\lesssim k\lesssim L_\Pi^{1/2}$) is resolved using estimate \eqref{fact1} (which holds independently of the joints structure). In particular,
\begin{eqnarray} \label{large k on plane}
   \begin{aligned}
      \sum_{Q\lesssim k\lesssim L_\Pi^{1/2}} |J_{k,\Pi}| k^{3/2}&=\sum_{Q_\Pi\lesssim k\lesssim L_\Pi^{1/2}} |J_{k,\Pi}| k^3 \frac{1}{k^{3/2}}\\
      &\lesssim L_{\Pi}^2\sum_{Q\lesssim k\lesssim L_\Pi^{1/2}} \frac{1}{k^{3/2}}\\
      &\lesssim L_\Pi^2 \frac{1}{Q^{3/2}}\sim L_\Pi^2\frac{L^{1/2}}{L_\Pi}\sim L_\Pi L^{1/2}.
   \end{aligned}
\end{eqnarray}
 Combining \eqref{small k on plane} and \eqref{large k on plane}, the desired estimate \eqref{interesting k on plane} follows.
\end{proof}

Now we prove \eqref{eq:discrete kakeya} in the  general case. For large $k$, 
the Szemer\'edi--Trotter theorem implies that
\begin{equation}\label{eq: bushes}
    |J_k|\lesssim \frac{L}{k}\text{ for all }k\gtrsim L^{1/2},
\end{equation}
hence
\begin{equation}\label{eq:trivial estimate}
    \sum_{k\gtrsim L^{1/2}} |J_k|k^{3/2}\lesssim \sum_{k\gtrsim L^{1/2}} L k^{1/2}\lesssim L^{3/2}.
\end{equation}
For (small) good $k$ the inequality
\begin{equation*}
    \sum_{\text{good }k}|J_k|k^{3/2}\lesssim L^{3/2}
\end{equation*}
follows from the superior estimate \eqref{eq:excepgood_k}. Finally, the estimate
\begin{equation*}
    \sum_{\text{bad }k}|J_k|k^{3/2}\lesssim L^{3/2},
\end{equation*}
follows directly from the fact that $J_\text{bad}$ has nearly planar structure together with Lemma~\ref{lemma:nearlyplanarest}. This completes the proof of Step 6 and of the theorem.

\end{proof}
\begin{remark}\label{large_l}
{\rm As is shown by the case where all lines in $\mathcal{L}$ pass through the same point, equality in \eqref{eq:trivial estimate} is sometimes (essentially) achieved. Note that, in this case of large $k$, \eqref{eq: bushes} implies via a simple counting argument (and independently of the joints structure) that the joints and lines are arranged in essentially non-interacting bushes.}
\end{remark}

\appendix 
\section{Appendix: Affine-invariant Hasse calculus -- directional derivatives, restrictions and multiplicities}\label{Appendix}

Throughout this appendix we will consider polynomials and not polynomial mappings. We shall regard all vectors in $\mathbb{F}^n$ as {\em column vectors} unless otherwise stated. 

\subsection{The Hasse derivative.} 

Let $n\geq 1$. For any $i=(i_1,\ldots,i_n)$ and $j=(j_1,\ldots,j_n)\in\mathbb{N}^n$, define
$$\binom{i}{j}:=\binom{i_1}{j_1}\cdots \binom{i_n}{j_n}.
$$
For all $i=1,\ldots,n$, denote by $e_i$ the vector $(0,\ldots,0,1,0,\ldots,0)^T$ with $1$ in the $i$-th coordinate. Finally, for any field $\mathbb{F}$, any $x=(x_1,\ldots,x_n)^T\in\mathbb{F}^n$ and any multi-index $a=(a_1,\ldots,a_n)\in\mathbb{N}^n$, let $x^a:=x_1^{a_1}\cdots x_n^{a_n}$.

Theorem~\ref{multijoints} is proved by studying directional derivatives of appropriate polynomials along directions carried by the objects forming the joints. While in a general field setting derivatives cannot be defined analytically, they can be defined algebraically as coefficients in Taylor expansions.

\begin{definition} \emph{\textbf{(Hasse derivative)}} Let $\mathbb{F}$ be a field, $n\geq 1$ and $p\in\mathbb{F}[x_1,\ldots,x_n]$. For each $a\in\mathbb{N}^n$, the \emph{Hasse derivative} $D^ap$ of $p$ is defined as the element of $\mathbb{F}[x_1,\ldots,x_n]$ which is the coefficient of $y^a$ in the expression of $p(x+y)$ as a polynomial in $x$.

\end{definition}

It follows that, for all $p\in\mathbb{F}[x_1,\ldots,x_n]$, we have the ``Taylor expansion"
$$p(x)=\sum_{a\in\mathbb{N}^n}D^ap(x_0)(x-x_0)^a
$$
in the sense of equality between polynomials in $x$ and $x_0$, and therefore also in the sense of polynomials in $x$ with $x_0\in \mathbb{F}^n$ fixed. Moreover, if we know that an expression 
$$p(x)=\sum_{a\in\mathbb{N}^n}p_a(x_0)(x-x_0)^a.
$$
with $p_a$ a polynomial in $x_0$ holds in the world of polynomials in $x$ and $x_0$, then we can deduce that $p_a =  D^ap$.\footnote{More precisely, this equality holds in $\mathbb{F}[x,x_0]$ (and thus in $\mathbb{F}[x_0]$). Indeed, suppose that $\sum_{a\in\mathbb{N}^n, |a| \leq N}q_a(x_0)(x-x_0)^a = 0$, with $q_a$ a polynomial in $x_0$, and that some $q_a$ with $|a|=N$ is non-zero.  The coeffcients of $x^a$ with $|a|=N$ must be zero, and hence $q_a = 0$ for all $a$ with $|a| = N$, contradiction.}
\begin{remark} {\rm {Observe that one can recover a polynomial via its Hasse derivatives at a point. In the special case where $\mathbb{F}=\mathbb{R}$, the Hasse derivative $D^ap(x_0)$ is simply a (non-zero) multiple of the usual derivative $(e_1\cdot\nabla)^{a_1}\cdots (e_n\cdot\nabla)^{a_n}p(x_0)$ of $p$ at $x_0$; more precisely,
$$(e_1\cdot\nabla)^{a_1}\cdots (e_n\cdot\nabla)^{a_n}p(x_0)=a!\; D^ap(x_0).$$
So, in this particular case the usual and Hasse derivatives are equivalent notions. However, in general field settings the ``usual" derivatives $a!D^ap(x_0)$ of a polynomial at a point provide less information about the polynomial, in that they do not suffice to fully recover the polynomial. For instance, all ``usual" derivatives of the polynomial $p(x)=x^q\in \mathbb{Z}_q[x]$ for $q$ prime vanish at 0, yet $p$ has a non-zero coefficient ($D^{(q)}p(0)=1\neq 0$). Therefore the Hasse derivative generalises the standard Euclidean space derivative in a more robust way than the ``usual" derivative does. In particular, even in the case $n=1$, it is not in general the case that $D^{(2)} = D^{(1)} \circ D^{(1)}$ (see Proposition~\ref{properties} (iii) below), and it is quite possible for a polynomial to satisfy $D^a p = 0$ while $D^{a+1} p \neq 0$ -- consider for example $p(x) = x^q$ in $\mathbb{Z}_q[x]$, with $a = q-1$.}}
\end{remark}

\begin{proposition} \label{properties} \emph{\textbf{(Basic properties of the Hasse derivative.)}} Let $\mathbb{F}$ be a field and $n\geq 1$. Then, the following hold:
\begin{enumerate}[{\rm (i)}]

\item For each $a$, $D^a: \mathbb{F}[x_1,\ldots,x_n] \to \mathbb{F}[x_1,\ldots,x_n]$ is a linear map.\\

\item For any monomial $x_1^{a_1}\cdots x_n^{a_n}\in\mathbb{F}[x_1,\ldots,x_n]$, it holds that
\begin{equation*}
    D^{e_i}(x^{a_1}\cdots x^{a_n})=
                \begin{cases}
                  a_i\; x_1^{a_1}\cdots x_i^{a_i-1}\cdots x_n^{a_n},&\text{ if }\ a_i> 0,\\
                  0,&\text{ if }\ a_i=0
                \end{cases}.
\end{equation*}

\item $D^i\left(D^jp\right)=\binom{i+j}{j}D^{i+j}p
= \binom{i+j}{i}D^{i+j}p = {D^j\left(D^ip\right)}$, for all $p\in\mathbb{F}[x_1,\ldots,x_n]$ and $i,j\in\mathbb{N}^n$.\\

\item Each $D^a$ is translation-invariant: $D^a(p( \cdot + y))(x) = D^ap( x + y)$ as polynomials in $x$ and $y$.
\end{enumerate}
\end{proposition}

Proofs for properties (i) and (iii) can be found for example in \cite{Dvir} and \cite{Sudan}, while (ii) is proved in \cite{Iliopoulou_13}. The proof of (iv) is an easy exercise.

Much of the rest of this section is devoted to a careful verification that calculus with Hasse derivatives proceeds in parallel with classical calculus. In subsequent subsections we consider, in turn, directional derivatives, restrictions of derivatives of polynomials to planes, Hasse-multiplicities of polynomials and vanishing properties of restrictions of directional derivatives of polynomials to planes. Many of the statements which follow also appear, in disguised form, in \cite{Zhang_16}.

The following technical lemma describes the derivatives of restrictions of polynomials to affine subspaces, and will subsequently be used for the study of directional derivatives.

\begin{lemma} \label{restricting on plane}Let $\mathbb{F}$ be a field, $n\geq 1$, $p\in\mathbb{F}[x_1,\ldots,x_n]$ and $x_0\in\mathbb{F}^n$. Let $k \in \{1,2, \dots , n\}$ and let $P$ be the $k$-plane through $x_0$ spanned by the vectors $\omega_1,\ldots,\omega_k\in\mathbb{F}^n\setminus\{0\}$. Let $\Omega$ be the $n \times k$ matrix with columns $\omega_1,\ldots,\omega_k$, and let, for $t=(t_1,\ldots,t_k)^T$,
\begin{equation*}
    p_{|_{P_{x_0}}}(t):=p(x_0+\Omega t)\in\mathbb{F}[t_1,\ldots,t_k].
\end{equation*}
Then, for all $m\in\mathbb{N}^k$, the identity 
\begin{equation*}
    D^m\big(p_{|_{P_{x_0}}}\big)(t)=\sum_{a=\alpha_1+\cdots+\alpha_k\in\mathbb{N}^n:\;|\alpha_i|=m_i\;\forall i}D^ap(x_0+\Omega t)\;\omega_1^{\alpha_1}\cdots\omega_k^{\alpha_k}
\end{equation*}
holds in $\mathbb{F}[t_1,\ldots,t_k]$.
\end{lemma}

\begin{proof} For convenience we denote the entries of $\Omega$ by $(\omega_{{jl}})_{j=1}^n{ }_{l=1}^{k}$, so that the column vector $\omega_l$ has entries ${(\omega_{{jl}})}_{j=1}^n$. For any vectors $t=(t_1,\ldots,t_k)^T$ and $t_0=(t_{01},\ldots,t_{0k})^T$ of indeterminants, we have
\begin{eqnarray*}
\begin{aligned}
p_{|_{P_{x_0}}}(t)&=p(x_0+\Omega t)\\
&=\sum_{a\in\mathbb{N}^n}D^ap(x_0+\Omega t_0)\cdot \big(\Omega(t-t_0)\big)^a\\
&=\sum_{a\in\mathbb{N}^n}D^ap(x_0+\Omega t_0)\cdot \big((t_1-t_{01})\;\omega_1+\cdots +(t_k-t_{0k})\;\omega_k\big)^a\\
&=\sum_{a\in\mathbb{N}^n}D^ap(x_0+\Omega t_0)\cdot \prod_{j=1}^n\big((t_1-t_{01})\;\omega_{j1}+\cdots +(t_k-t_{0k})\;\omega_{jk}\big)^{a_j}.
\end{aligned}
\end{eqnarray*}
With $a \in \mathbb{N}^n$ and $j$ fixed we have
$$
\big((t_1-t_{01})\;\omega_{j1}+\cdots  +(t_k-t_{0,k})\;\omega_{jk}\big)^{a_j}
= \sum_{b_{j1} + \cdots + b_{jk} = a_j} 
\left[(t_{1} -t_{01}) \omega_{j1}\right]^{b_{j1}} \cdots 
\left[(t_{k} -t_{0k}) \omega_{jk}\right]^{b_{jk}}
 $$
 and the product in $j$ of these terms is therefore
 $$ \prod_{j=1}^n \sum_{b_{j1} + \cdots + b_{jk} = a_j} 
\left[(t_{1} -t_{01}) \omega_{j1}\right]^{b_{j1}} \cdots 
\left[(t_{k} -t_{0k}) \omega_{jk}\right]^{b_{jk}}.
 $$
 With $a$ still fixed, let $B$ be the $n\times k$ matrix whose entries are $b_{jl}$. Denote its rows by $b_j \in \mathbb{N}^k$ and its columns by $\alpha_l \in \mathbb{N}^n$, so that for each $j$ the entries $b_{jl}$ of $b_j$ satisfy $\sum_{l=1}^k b_{jl} =a_j$. The previous displayed expression becomes
 \begin{eqnarray*}
    \begin{aligned}
        \sum_{\alpha_1 + \cdots + \alpha_k = a}&\left[(t_{1} -t_{01})^{\sum_{j=1}^n b_{j1}} \prod_{j=1}^n\omega_{j1}^{b_{j1}}\right] \cdots 
        \left[(t_{k} -t_{0k})^{\sum_{j=1}^n b_{jk}} \prod_{j=1}^n \omega_{jk}^{b_{jk}}\right]\\
       &= \sum_{\alpha_1 + \cdots + \alpha_k = a} (t-t_0)^{\sum_{j=1}^n b_j} \omega_1^{\alpha_1} \cdots \omega_k^{\alpha_k}\\
       & = \sum_{m \in \mathbb{N}^k} (t-t_0)^m\sum_{\alpha_1 + \dots + \alpha_k = a, \sum_{j=1}^n b_j = m} \omega_1^{\alpha_1} \cdots \omega_k^{\alpha_k}\\
       &= \sum_{m \in \mathbb{N}^k} (t-t_0)^m\sum_{a= \alpha_1 + \dots + \alpha_k, |\alpha_l|= m_l  \; \forall \; l} \omega_1^{\alpha_1} \cdots \omega_k^{\alpha_k}.
    \end{aligned}
 \end{eqnarray*}
Therefore, summing over $a$,
 \begin{eqnarray*}
\begin{aligned}
p_{|_{P_{x_0}}}(t)
=\sum_{m\in\mathbb{N}^k}\left(\sum_{a=\alpha_1+\cdots +\alpha_k\in\mathbb{N}^n:\;|\alpha_i|=m_i\;\forall \;i} D^ap(x_0+\Omega t_0)\;\omega_1^{\alpha_1}\cdots \omega_k^{\alpha_k}\right)(t-t_0)^m.
\end{aligned}
\end{eqnarray*}
Consequently, for any $m\in\mathbb{N}^k$, $D^m\big(p_{|_{P_{x_0}}}\big)(t_0)$ equals the coefficient of $(t-t_0)^m$ in the last expression above, and we are done.

\end{proof}

\subsection{Directional derivatives.}
As with standard derivatives, directional derivatives can be understood algebraically in Euclidean space and can therefore be meaningfully defined in all field settings. In particular, it is easy to see that for all linearly independent vectors $\omega_1,\ldots,\omega_n$ in $\mathbb{R}^n$ and any $a=(a_1,\ldots,a_n)\in\mathbb{N}^n$, it holds that $$(\omega_1\cdot\nabla)^{a_1}\cdots (\omega_n\cdot\nabla)^{a_n}p(x_0)=a!\; D^a(p\circ L)(L^{-1}x_0)
$$
where $L:\mathbb{R}^n\rightarrow\mathbb{R}^n$ is the linear isomorphism with $L(e_i)=\omega_i$. This observation leads to the following definition:

\begin{definition} \emph{\textbf{(Directional Hasse derivative.)}} Let $\mathbb{F}$ be a field, $n\geq 1$ and $p\in\mathbb{F}[x_1,\ldots,x_n]$. Suppose that $\omega_1,\ldots,\omega_n\in\mathbb{F}^n$ are linearly independent vectors, and let $L:\mathbb{F}^n\rightarrow\mathbb{F}^n$ be the linear isomorphism with $L(e_i)=\omega_i$ for all $i=1,\ldots,n$. For each $a=(a_1,\ldots,a_n)\in \mathbb{N}^n$, we define
\begin{equation*}
(\omega_1\cdot\nabla)^{a_1}\cdots (\omega_n\cdot\nabla)^{a_n}p(x):=D^a(p\circ L)(L^{-1}x)\in \mathbb{F}[x_1,\ldots,x_n].
\end{equation*}
\end{definition}
Sometimes we write this more succinctly as 
\begin{equation*}
(\boldsymbol{\omega}\cdot\nabla)^ap:=(\omega_1\cdot\nabla)^{a_1}\cdots (\omega_n\cdot\nabla)^{a_n}p
\end{equation*}
where $\boldsymbol{\omega}:=(\omega_1,\ldots,\omega_n)$. 
Note that, for any $a=(a_1,\ldots,a_n)\in\mathbb{N}^n$, this definition introduces the alternative notation $(e_1\cdot\nabla)^{a_1}\cdots (e_n\cdot\nabla)^{a_n}p(x_0)$ for $D^ap(x_0)$.  

A directional derivative can easily be expressed in terms of standard Hasse derivatives, and more generally in terms of directional derivatives in another set of fixed directions, as follows.

\begin{lemma} \label{directional in terms of standard}
Let $\mathbb{F}$ be a field, $n\geq 1$ and $p\in\mathbb{F}[x_1,\ldots,x_n]$. For any linearly independent vectors $\omega_1,\ldots,\omega_n\in\mathbb{F}^n$, for any $(a_1,\ldots,a_n)\in \mathbb{N}^n$, the equality
\begin{equation*}
    (\omega_1\cdot\nabla)^{a_1}\cdots (\omega_n\cdot\nabla)^{a_n}p(x)=\sum_{\tilde{a}=\alpha_1+\cdots+\alpha_n\in\mathbb{N}^n:|\alpha_i|=a_i\;\forall i}D^{{\tilde{a}}}p(x)\;\omega_1^{\alpha_1}\cdots\omega_n^{\alpha_n}
\end{equation*}
holds in $\mathbb{F}[x_1,\ldots,x_n]$.
\end{lemma}

This is a simple application of Lemma~\ref{restricting on plane} in the case $k=n$ for the polynomial $(\omega_1\cdot\nabla)^{a_1}\cdots (\omega_n\cdot\nabla)^{a_n}p$, and easily implies the more general identity
\begin{equation}\label{leading to multiplicity invariance}
    (\boldsymbol{\omega}\cdot \nabla)^ap(x)=\sum_{\tilde{a}=\alpha_1+\cdots+\alpha_n\in\mathbb{N}^n:|\alpha_i|=a_i\;\forall i}(\boldsymbol{\overline{\omega}}\cdot \nabla)^{\tilde{a}}p(x)\;\widetilde{\omega}_1^{\alpha_1}\cdots\widetilde{\omega}_n^{\alpha_n}
\end{equation}
in $\mathbb{F}[x_1,\ldots,x_n]$, for all $n$-tuples $\boldsymbol{\omega}=(\omega_1,\ldots,\omega_n)$ and $\boldsymbol{\overline{\omega}}=(\overline{\omega_1},\ldots,\overline{\omega_n})$ of linearly independent vectors in $\mathbb{F}^n$, where for each $j$, $\widetilde{\omega}_j=L\overline{L}^{-1}(e_j)$, where $L$ is the linear isomorphism of $\mathbb{F}^n$ sending each $e_i$ to $\omega_i$, and $\overline{L}$ the linear isomorphism of $\mathbb{F}^n$ sending each $e_i$ to $\overline{\omega_i}$.

\begin{remark} \label{transverse derivative definition}
{\rm Let $1\leq k\leq n$,  $(a_{k+1},\ldots,a_n)\in\mathbb{N}^{n-k}$ and let $\omega_{k+1},\ldots,\omega_n\in \mathbb{F}^n$ be linearly independent. The above lemma implies that the polynomial
\begin{equation}\label{eq:transverse derivative}
    (\omega_1\cdot\nabla)^0\cdots (\omega_k\cdot\nabla)^0\cdot (\omega_{k+1}\cdot\nabla)^{a_{k+1}}\cdots (\omega_n\cdot\nabla)^{a_n}p\in\mathbb{F}[x_1,\ldots,x_n]
\end{equation}
is independent of the choice of vectors $\omega_1,\ldots,\omega_k\in\mathbb{F}^n$ with the property that ${\rm span}\{\omega_1,\ldots,\omega_n\}=\mathbb{F}^n$, as one would expect. We thus henceforth denote any polynomial in \eqref{eq:transverse derivative} by
\begin{equation*}
    (\omega_{k+1}\cdot\nabla)^{a_{k+1}}\cdots (\omega_n\cdot\nabla)^{a_n}p.
\end{equation*}
It follows that
\begin{equation*}
    (\omega_{k+1}\cdot\nabla)^{a_{k+1}}\cdots (\omega_n\cdot\nabla)^{a_n}p(x)=D^a(p\circ L)(L^{-1}x)
\end{equation*}
 where $a:=(0,\ldots,0,a_{k+1},\ldots,a_n)\in\mathbb{N}^n$, for all linear isomorphisms $L:\mathbb{F}^n\rightarrow \mathbb{F}^n$ such that $L(e_i)=\omega_i$ for $i=k+1,\ldots,n$.
}
\end{remark}

It will be seen that directional derivatives enjoy to a large extent properties analogous to those of standard directional derivatives in Euclidean space.

\subsection{Restrictions of derivatives of polynomials to planes.}
Restrictions of directional Hasse derivatives of polynomials to planes can be themselves viewed as polynomials in a natural way.

\begin{definition} \label{def: transverse directions}
{\rm Let $\mathbb{F}$ be a field, $n\geq 1$, $1\leq k\leq n$ and $p\in\mathbb{F}[x_1,\ldots,x_n]$. Let $P=x_0+{\rm span}\{\omega_1,\ldots,\omega_k\}$ be a $k$-dimensional plane in $\mathbb{F}^n$. We say that the vectors $\omega_{k+1},\ldots,\omega_n\in\mathbb{F}^n$ are \emph{transverse to }$P$ if, together with $\omega_1,\ldots,\omega_k$, they span $\mathbb{F}^n$.}
\end{definition}

\begin{definition} \emph{\textbf{(Restrictions of directional derivatives of polynomials to planes.)}} \label{restriction of dhd on a plane} {\rm Let $\mathbb{F}$ be a field, $n\geq 1$, $1\leq k\leq n$ and $p\in\mathbb{F}[x_1,\ldots,x_n]$. Let $P=x_0+{\rm span}\{\omega_1,\ldots,\omega_k\}$ be a $k$-dimensional plane in $\mathbb{F}^n$. Let $\Omega$ be the $n\times k$ matrix with columns $\omega_1,\ldots,\omega_k$. For any vectors $\omega_{k+1},\ldots,\omega_n\in\mathbb{F}^n$ transverse to $P$ and for any $a=(a_1,\ldots,a_n)\in\mathbb{N}^n$, define the polynomial 
\begin{equation*}
    (\omega_1\cdot\nabla)^{a_1}\cdots (\omega_n\cdot\nabla)^{a_n}p_{|_{P_{x_0,\Omega}}}\in \mathbb{F}[t_1,\ldots,t_k]
\end{equation*}
by
\begin{equation*}
(\omega_1\cdot\nabla)^{a_1}\cdots (\omega_n\cdot\nabla)^{a_n}p_{|_{P_{x_0,\Omega}}}(t_1,\ldots,t_k):=(\omega_1\cdot\nabla)^{a_1}\cdots (\omega_n\cdot\nabla)^{a_n}p(x_0+t_1\omega_1+\cdots +t_k\omega_k).
\end{equation*}
}
\end{definition}

In the proof of Theorem~\ref{multijoints}, for $p\in\mathbb{F}[x_1,\ldots,x_n]$ we employ the notation $p_{|_{P}}$ to denote the standard restriction of the function $p:\mathbb{F}^n\rightarrow\mathbb{F}$ to $P$. In this appendix, however, the more elaborate notation $p_{|_{P_{x_0,\Omega}}}$ (and also the form $p_{|_{P_{x_0}}}$ used in Lemma~\ref{restricting on plane}) is reserved to denote the polynomial in $\mathbb{F}[t_1,\ldots,t_k]$ given by the above definition. (Observe that, by Lemma~\ref{infinite}, in the case where $\mathbb{F}$ is infinite, $p_{|_{P_{x_0,\Omega}}}$ is the zero polynomial in $\mathbb{F}[t_1,\ldots,t_k]$ if and only if the function $p_{|_{P}}$ is zero. This property is independent of the particular $x_0$ and $\Omega$ used to define $P$. Likewise, ${\rm deg} \, p_{|_{P_{x_0, \Omega}}}$, and the multiplicity of $p_{|_{P_{x_0, \Omega}}}$ (which will be discussed in Lemma~\ref{restriction properties}) at any point of $\mathbb{F}^k$, are 
independent of the particular $x_0, \Omega$ used to define $P$.)

\begin{remark} \label{remark on restrictions}
{\rm 
Using the above notation, and recalling that
\begin{equation*}
    (\omega_1\cdot\nabla)^{a_1}\cdots(\omega_n\cdot\nabla)^{a_n}p(x):=D^a(p\circ L)(L^{-1}x)
\end{equation*}
where $L:\mathbb{F}^n\rightarrow\mathbb{F}^n$ is the linear isomorphism with $L(e_i)=\omega_i$ for all $i=1,\ldots,n$, it follows that
\begin{eqnarray*} \label{eq:isomorphism language for restriction}
\begin{aligned}
(\omega_1\cdot\nabla)^{a_1}\cdots (\omega_n\cdot\nabla)^{a_n}p_{|_{P_{x_0,\Omega}}}(t_1,\ldots,t_k)&=D^a(p\circ L)[L^{-1}(x_0+t_1\omega_1+\cdots +t_k\omega_k)]\\
&=D^a(p\circ L)(L^{-1}x_0+t_1e_1+\cdots +t_ke_k).
\end{aligned}
\end{eqnarray*}
Of particular interest to us will be restrictions to $P$ of directional derivatives of the form
\begin{equation*}
    (\omega_{k+1}\cdot\nabla)^{a_{k+1}}\cdots (\omega_n\cdot\nabla)^{a_n}p,
\end{equation*}
i.e. derivatives in directions transverse to $P$. Recall that by Remark~\ref{transverse derivative definition} the equality
\begin{equation*}
    (\omega_{k+1}\cdot\nabla)^{a_{k+1}}\cdots (\omega_n\cdot\nabla)^{a_n}p=(\omega_1\cdot\nabla)^0\cdots(\omega_k\cdot\nabla)^0(\omega_{k+1}\cdot\nabla)^{a_{k+1}}\cdots (\omega_n\cdot\nabla)^{a_n}p
\end{equation*}
holds, hence
\begin{eqnarray*} 
   \begin{aligned}
        (\omega_{k+1}\cdot\nabla)^{a_{k+1}}\cdots (\omega_n\cdot\nabla)^{a_n}p_{|_{P_{x_0,\Omega}}}(t_1,\ldots,t_k)=D^a(p\circ L)(L^{-1}x_0+t_1e_1+\cdots +t_ke_k)
   \end{aligned}
\end{eqnarray*}
for the above isomorphism $L$ and for $a=(0,\ldots,0,a_{k+1},\ldots,a_n)$.
}
\end{remark}

\begin{lemma} \label{lemma derivatives of restrictions}
Let $\mathbb{F}$ be a field, $n\geq 1$, $1\leq k\leq n$ and $p\in\mathbb{F}[x_1,\ldots,x_n]$. Let $P=x_0+{\rm span}\{\omega_1,\ldots,\omega_k\}$ be a $k$-dimensional plane in $\mathbb{F}^n$ and denote by $\Omega$ the $n\times k$ matrix with columns $\omega_1,\ldots,\omega_k$. For every $\omega_{k+1},\ldots,\omega_n\in\mathbb{F}^n$ transverse to $P$ and all $a=(a_1,\ldots,a_n)\in\mathbb{N}^{n}$, the equality
\begin{equation*}D^{(a_1,\ldots,a_k)}\big[(\omega_{k+1}\cdot\nabla)^{a_{k+1}}\cdots (\omega_n\cdot\nabla)^{a_n}p_{|_{P_{x_0,\Omega}}}\big]=(\omega_1\cdot\nabla)^{a_1}\cdots (\omega_n\cdot\nabla)^{a_n}p_{|_{P_{x_0,\Omega}}}.
\end{equation*}
holds in $\mathbb{F}[t_1,\ldots,t_k]$.
\end{lemma}

\begin{proof}
Let $L:\mathbb{F}^n\rightarrow \mathbb{F}^n$ be the linear isomorphism with $L(e_i)=\omega_i$ for all $i=1,\ldots,n$. The statement of the lemma is that
\begin{equation*}
    D^{(a_1,\ldots,a_k)}\big[(\omega_{k+1}\cdot\nabla)^{a_{k+1}}\cdots (\omega_n\cdot\nabla)^{a_n}p_{|_{P_{x_0,\Omega}}}\big](t)=(\omega_1\cdot\nabla)^{a_1}\cdots (\omega_n\cdot\nabla)^{a_n}p(x_0+\Omega t)
\end{equation*}
in $\mathbb{F}[t_1,
\ldots,t_k]$, i.e. that the polynomial
\begin{eqnarray*}
    \begin{aligned}
        g(t_1,\ldots,t_k):&=(\omega_{k+1}\cdot\nabla)^{a_{k+1}}\cdots (\omega_n\cdot\nabla)^{a_n}p_{|_{P_{x_0,\Omega}}}(t_1,\ldots,t_k)\\
        &=(\omega_{k+1}\cdot\nabla)^{a_{k+1}}\cdots (\omega_n\cdot\nabla)^{a_n}p(x_0+t_1\omega_1+\cdots+t_k\omega_k)\\
        &= D^{(0,\ldots,0,a_{k+1},\ldots,a_n)}(p\circ L)(L^{-1}x_0+t_1e_1+\cdots +t_ke_k)\in\mathbb{F}[t_1,\cdots,t_k]
    \end{aligned}
\end{eqnarray*}
satisfies
\begin{equation*}
    D^{(a_1,\ldots,a_k)}g(t_1,\ldots,t_k)=D^a(p\circ L)(L^{-1}x_0+t_1e_1+\cdots +t_ke_k).
\end{equation*}

By Lemma~\ref{restricting on plane},
$$D^{(a_1,\ldots,a_k)}g(t_1,\ldots,t_k)$$
$$=\sum_{a'=\alpha_1+\cdots +\alpha_k\in\mathbb{N}^n:\; |\alpha_i|=a_i}D^{a'}\big[D^{(0,\ldots,0,a_{k+1},\ldots,a_n)}(p\circ L)\big](L^{-1}x_0+t_1e_1+\cdots +t_ke_k)\;e_1^{\alpha_1}\cdots e_k^{\alpha_k}.
$$
Now, for each $i\in\{1,\ldots,k\}$, $e_i^{\alpha_i}$ equals 0 unless $\alpha_i=(0,\ldots,0,a_i,0,\ldots,0)$, with $a_i$ in the $i$-th coordinate. Therefore, only one term survives in the sum, and we have
\begin{eqnarray*}
\begin{aligned}
D^{(a_1,\ldots,a_k)}g(t)&=D^{(a_1,\ldots,a_k,0,\ldots,0)}\big[D^{(0,\ldots,0,a_{k+1},\ldots,a_n)}(p\circ L)\big](L^{-1}x_0+t_1e_1+\cdots +t_ke_k)\\
&=D^{a}(p\circ L)(L^{-1}x_0+t_1e_1+\cdots +t_ke_k)
\end{aligned}
\end{eqnarray*}
as required, where the last equality is due to property (iii) of Hasse derivatives.
\end{proof}

\subsection{Multiplicities of polynomials} We now turn to the notion of multiplicity (or order of vanishing) of a polynomial at a point. Our subsequent analysis will rely upon this notion. The definition of multiplicity for Euclidean space carries over directly to the setting of arbitrary fields when we use the Hasse derivative. In this subsection, let $\mathbb{F}$ be a field, and $n\geq 1$.

\begin{definition} \emph{\textbf{(Multiplicity.)}} {\rm Let $p\in\mathbb{F}[x_1,\ldots,x_n]$ and $x_0\in\mathbb{F}^n$. If $p\neq 0$, the \emph{multiplicity of $p$ at $x_0$}, denoted by ${\rm mult}(p,x_0)$, is the largest $m\in\mathbb{N}$ with the property that $D^ap(x_0)=0$ for all $a\in\mathbb{N}^n$ with $|a|<m$.
If $p(x_0) \neq 0$ we say that ${\rm mult}(p,x_0) = 0$. If $p=0$ we say that ${\rm mult}(p,x_0)=+\infty$ for all $x_0\in\mathbb{F}^n$.
}
\end{definition}

\begin{definition} \emph{\textbf{(Directional multiplicity.)}} 
{\rm Let $p\in\mathbb{F}[x_1,\ldots,x_n]$.
If $p\neq 0$, for an $n$-tuple $\boldsymbol{\omega}$ of linearly independent vectors in $\mathbb{F}^n$, define the \emph{directional multiplicity ${\rm mult}_{\boldsymbol{\omega}}(p,x_0)$} of $p$ at $x_0\in\mathbb{F}^n$} to be the largest $m\in\mathbb{N}$ with the property that $(\boldsymbol{\omega}\cdot \nabla)^ap(x_0)= 0$ for all $a\in\mathbb{N}^n$ with $|a|<m$. (We make the obvious modifications if $p(x_0)\neq 0$ or $p=0$.)

\end{definition}
Central to our analysis is the following proposition, which states that the multiplicity of a polynomial at a point is independent of the choice of coordinate system, and is a direct consequence of \eqref{leading to multiplicity invariance}. 

\begin{proposition} \emph{\textbf{(Multiplicity invariance.)}}\label{multiplicity invariance} Let $p\in\mathbb{F}[x_1,\ldots,x_n]$. For any $x_0\in\mathbb{F}^n$ and any linearly independent vectors $\omega_1,\ldots,\omega_n$ in $\mathbb{F}^n$, it holds that
\begin{equation*}
    {\rm mult}_{\omega_1,\ldots,\omega_n}(p,x_0)={\rm mult}(p,x_0).
\end{equation*}

\end{proposition}

\subsection{Vanishing properties of restrictions of directional derivatives of polynomials to planes}

Counting points on a plane can be carried out using polynomials that vanish at the points of interest, but not identically on the plane. For a polynomial in $\mathbb{F}[x_1,\ldots,x_n]$ and a $k$-plane in $\mathbb{F}^n$, the following lemma facilitates the identification of directional derivatives with non-zero restrictions (when viewed as polynomials) on the plane. Under certain conditions, it also provides meaningful information on the order of vanishing of such restrictions.

\begin{lemma} \label{restriction properties}
Let $\mathbb{F}$ be a field, $n\geq 1$, $1\leq k\leq n$ and $p\in\mathbb{F}[x_1,\ldots,x_n]$. Let $P=x_0+{\rm span}\{\omega_1,\ldots,\omega_k\}$ be a $k$-dimensional plane in $\mathbb{F}^n$, and denote by $\Omega$ the $n\times k$ matrix with columns $\omega_1,\ldots,\omega_k$. Let $\omega_{k+1},\ldots,\omega_n$ be vectors in $\mathbb{F}^n$ transverse to $P$ and let $a=(a_1,\ldots,a_n)\in\mathbb{N}^{n}$.

\begin{enumerate}[{\rm (i)}]
\item \emph{\textbf{(Identifying derivatives with non-zero restrictions)}} If $(\omega_1\cdot\nabla)^{a_1}\cdots (\omega_n\cdot\nabla)^{a_n}p(x_0)\neq 0$, then
\begin{equation*}(\omega_{k+1}\cdot\nabla)^{a_{k+1}}\cdots (\omega_n\cdot\nabla)^{a_n}p_{|_{P_{x_0,\Omega}}}\neq 0.
\end{equation*}
\item \emph{\textbf{(Multiplicities)}} For all $y=x_0+\Omega t\in P$, it holds that
\begin{equation*}{\rm mult}\left( (\omega_{k+1}\cdot\nabla)^{a_{k+1}}\cdots (\omega_n\cdot\nabla)^{a_n}p_{|_{P_{x_0,\Omega}}},t\right)\geq{\rm mult}(p,y)-(a_{k+1}+\cdots +a_n).
\end{equation*}
\end{enumerate}
\end{lemma}

\begin{proof} Recall by Lemma~\ref{lemma derivatives of restrictions} that
\begin{equation} \label{derivatives of restrictions}
    D^{(a_1,\ldots,a_k)}\big[(\omega_{k+1}\cdot\nabla)^{a_{k+1}}\cdots (\omega_n\cdot\nabla)^{a_n}p_{|_{P_{x_0,\Omega}}}\big](t)=(\omega_1\cdot\nabla)^{a_1}\cdots (\omega_n\cdot\nabla)^{a_n}p(x_0+\Omega t)
\end{equation}
in $\mathbb{F}[t_1,\ldots,t_k]$. It follows that if $(\omega_1\cdot\nabla)^{a_1}\cdots (\omega_n\cdot\nabla)^{a_n}p(x_0)\neq 0$ then also \begin{equation*}
    D^{(a_1,\ldots,a_k)}\big[(\omega_{k+1}\cdot\nabla)^{a_{k+1}}\cdots(\omega_n\cdot\nabla)^{a_n}p_{|_{P_{x_0,\Omega}}}\big](0)\neq 0,
\end{equation*}
therefore $(\omega_{k+1}\cdot\nabla)^{a_{k+1}}\cdots(\omega_n\cdot\nabla)^{a_n}p_{|_{P_{x_0,\Omega}}}$ is not the zero polynomial. This establishes (i).

Furthermore, \eqref{derivatives of restrictions} implies that for any $t\in\mathbb{F}^k$
\begin{equation*}
    D^{(a_1',\ldots,a_k')}\big[(\omega_{k+1}\cdot\nabla)^{a_{k+1}}\cdots (\omega_n\cdot\nabla)^{a_n}p_{|_{P_{x_0,\Omega}}}\big](t)=0
\end{equation*}
for all $(a_1',\ldots,a_k')\in\mathbb{N}^k$ with 
\begin{eqnarray*}
   \begin{aligned}
       (a_1'+\ldots+a_k')+ (a_{k+1}+\cdots+a_n) &< {\rm mult}_{\omega_1,\ldots,\omega_n}(p,x_0+\Omega t)\\
       &={\rm mult}(p,x_0+\Omega t),
   \end{aligned}
\end{eqnarray*}
thereby directly implying (ii).

\end{proof}
Assertion (i) can be used to identify derivatives of $p$ with non-zero restrictions (when viewed as polynomials) to a plane $P$. Now, let $g:=(\omega_{k+1}\cdot\nabla)^{a_{k+1}}\cdots (\omega_n\cdot\nabla)^{a_n}p$ be such a derivative. This implies that for every $y\in P$ there exists $(a_1',\ldots,a_k')\in\mathbb{N}^k$ with
\begin{equation*}
    (\omega_1\cdot\nabla)^{a_1'}\cdots(\omega_k\cdot\nabla)^{a_k'}g(y)\neq 0,
\end{equation*}
and thus by (ii), ${\rm mult}(g,y)\geq {\rm mult}(p,y)-(a_{k+1}+\cdots +a_n)$, a quantity which may well be non-positive. And indeed, in general there can be no guarantee that $g$ vanishes at points $y\in P$ of interest (such as in the case where $p$ is a non-zero constant polynomial and $g=p$). 

If however $g$ is a derivative \textit{of minimal order} that does not identically on $P$, it transpires that the quantity ${\rm mult}(p,y)-(a_{k+1}+\cdots +a_n)$ is nonnegative, and is in fact positive under suitable conditions. This is made precise in Lemma~\ref{minimal derivative} below. 

\begin{remark} \label{considering different directions}
{\rm Observe that any directional derivative of $p$ of minimal order with non-zero restriction (when viewed as a polynomial) on $P$ is necessarily a derivative in directions transverse to $P$, i.e. it is of the form
\begin{equation*}
    (\overline{\omega_{k+1}}\cdot\nabla)^{m_{k+1}}\cdots (\overline{\omega_n}\cdot\nabla)^{m_n}p
\end{equation*}
where $\overline{\omega_{k+1}},\ldots,\overline{\omega_n}$ are vectors in $\mathbb{F}^n$ which, together with $\omega_1,\ldots,\omega_k$, span $\mathbb{F}^n$. Moreover, it follows by Lemma~\ref{restriction properties} (i) that if $p$ is not the zero polynomial, then, for any directions transverse to $P$, there exists a derivative of $p$ in these directions whose restriction to $P$ is not the zero polynomial.}

\end{remark}

\begin{lemma} \label{minimal derivative}Let $\mathbb{F}$ be a field, $n\geq 1$, $1\leq k\leq n$ and $p\in\mathbb{F}[x_1,\ldots,x_n]$ be non-zero. Let $P=x_0+{\rm span}\{\omega_1,\ldots,\omega_k\}$ be a $k$-dimensional plane in $\mathbb{F}^n$, and denote by $\Omega$ the $n\times k$ matrix with columns $\omega_1,\ldots,\omega_k$. Let $y=x_0+\Omega t\in P$. Fix $\omega_{k+1},\ldots,\omega_n\in\mathbb{F}^n$ transverse to $P$ and let $a=(a_1,\ldots,a_n)\in\mathbb{N}^n$ be of minimal length such that
\begin{equation*}
(\omega_1\cdot\nabla)^{a_1}\cdots (\omega_n\cdot\nabla)^{a_n}p(y)\neq 0.
\end{equation*}
Then, any directional derivative $\mathcal{D}p$ of $p$ of minimal order such that $\mathcal{D}p_{|_{P_{x_0,\Omega}}}\neq 0$ satisfies
\begin{equation*}
    {\rm mult}\left(\mathcal{D}p_{|_{P_{x_0,\Omega}}},t\right)\geq a_1+\cdots +a_k.
\end{equation*}
\end{lemma}

\begin{proof}
Let $\mathcal{D}p$ be a directional derivative of $p$ of minimal order such that $\mathcal{D}p_{|_{P_{x_0,\Omega}}}\neq 0$ in $\mathbb{F}[t_1,\ldots,t_k]$. Then,
\begin{equation*}
    \mathcal{D}p=(\overline{\omega_{k+1}}\cdot\nabla)^{m_{k+1}}\cdots (\overline{\omega_n}\cdot\nabla)^{m_n}p
\end{equation*} for some $m=(m_{k+1},\ldots,m_n)\in\mathbb{N}^{n-k}$ and $\overline{\omega_{k+1}},\ldots,\overline{\omega_n}\in\mathbb{F}^n$ transverse to $P$. Fix $y=x_0+\Omega t\in\mathbb{F}^n$. Let $\omega_{k+1},\ldots,\omega_n\in\mathbb{F}^n$ be vectors transverse to $P$ and $a=(a_1,\ldots,a_n)\in\mathbb{N}^n$ be of minimal length such that 
\begin{equation*}
    (\omega_1\cdot\nabla)^{a_1}\cdots (\omega_n\cdot\nabla)^{a_n}p(y)\neq 0.
\end{equation*}
It follows by Lemma~\ref{restriction properties} that
\begin{equation*}
    (\omega_{k+1}\cdot\nabla)^{a_{k+1}}\cdots(\omega_n\cdot\nabla)^{a_n}p_{|_{P_{y,\Omega}}}\neq 0
\end{equation*}
or equivalently that
\begin{equation*}
    (\omega_{k+1}\cdot\nabla)^{a_{k+1}}\cdots(\omega_n\cdot\nabla)^{a_n}p_{|_{P_{x_0,\Omega}}}\neq 0.
\end{equation*}
The minimality property of $m$ implies that $m_{k+1}+\cdots +m_n\leq a_{k+1}+\cdots +a_n$. On the other hand, the minimality property of $a$ implies that $a_1+\cdots +a_n={\rm mult}_{\omega_1,\ldots,\omega_n}(p,y)$, therefore
\begin{eqnarray*}
    \begin{aligned}
       m_{k+1}+\ldots m_n&\leq (a_1+\ldots+a_n)-(a_1+\cdots+a_k)\\
       &={\rm mult}_{\omega_1,\ldots,\omega_n}(p,y)-(a_1+\cdots+a_k)\\
       &={\rm mult}(p,y)-(a_1+\cdots +a_k).
    \end{aligned}
\end{eqnarray*}
Combining assertion (ii) of Lemma~\ref{restriction properties} with the above, one deduces that
\begin{eqnarray*}
\begin{aligned}
{\rm mult}\left(\mathcal{D}p,t\right)&\geq{\rm mult}(p,y)-(m_{k+1}+\cdots +m_n)\\
&\geq a_1+\cdots +a_k,
\end{aligned}
\end{eqnarray*}
as required.

\end{proof}

Let $p\in\mathbb{F}[x_1,\ldots,x_n]$ be a non-zero polynomial. Lemma~\ref{restriction properties} (i) asserts that, for any plane $P$, if one takes enough derivatives of $p$ in directions transverse to $P$, then the resulting polynomial will not vanish identically on $P$. The following lemma (which is a direct consequence of \eqref{leading to multiplicity invariance}) states that the number of derivatives required to achieve this is independent of the directions along which we choose to differentiate.

\begin{lemma} \label{order invariance of minimal derivative}
Let $\mathbb{F}$ be a field, $n\geq 1$ and $p\in\mathbb{F}[x_1,\ldots,x_n]$ be non-zero. Let $1\leq k\leq n$ and $P:=x_0+{\rm span}\{\omega_1,\ldots,\omega_k\}$ be a $k$-dimensional plane in $\mathbb{F}^n$. Let $\omega_{k+1},\ldots,\omega_n\in\mathbb{F}^n$ be transverse to $P$, and let
\begin{equation*}
    \mathcal{D}p:=(\omega_{k+1}\cdot\nabla)^{m_{k+1}}\cdots (\omega_n\cdot\nabla)^{m_n}p
\end{equation*}
be a derivative of $p$ with the property that, amongst all derivatives of $p$ in directions $\omega_{k+1},\ldots,\omega_n$, $\mathcal{D}p$ is of minimal order so that $\mathcal{D}p_{|_{P_{x_0,\Omega}}}\neq 0$ in $\mathbb{F}[t_1,\ldots,t_k]$. Furthermore, let $\overline{\omega_{k+1}},\ldots,\overline{\omega_n}\in\mathbb{F}^n$ be transverse to $P$, and let
\begin{equation*}
    \overline{\mathcal{D}}p:=(\overline{\omega_{k+1}}\cdot\nabla)^{\lambda_{k+1}}\cdots (\overline{\omega_n}\cdot\nabla)^{\lambda_n}p
\end{equation*}
be a derivative of $p$ with the property that, amongst all derivatives of $p$ in directions $\overline{\omega_{k+1}},
\ldots,\overline{\omega_n}$, $\overline{\mathcal{D}}p$ is of minimal order such that  and $\overline{\mathcal{D}}p_{|_{P_{x_0,\Omega}}}\neq 0$ in $\mathbb{F}[t_1,\ldots,t_k]$. Then,
\begin{equation*}
    m_{k+1}+\cdots +m_n=\lambda_{k+1}+\cdots +\lambda_n.
\end{equation*}
\end{lemma}

\section*{Acknowledgments} 
The authors are grateful to the anonymous referees for suggestions that greatly improved the exposition of the paper.

\bibliographystyle{amsplain}


\begin{dajauthors}
\begin{authorinfo}[pgom]
  Anthony Carbery\\
  Professor\\
  University of Edinburgh\\
 School of Mathematics and Maxwell Institute for Mathematical Sciences\\
  Edinburgh, UK\\
  A.Carbery\imageat{}ed\imagedot{}ac\imagedot{}uk \\
  \url{https://www.maths.ed.ac.uk/school-of-mathematics/people/a-z?person=61}
\end{authorinfo}
\begin{authorinfo}[johan]
  Marina Iliopoulou\\
  Lecturer\\
  University of Kent\\
  School of Mathematics, Statistics and Actuarial Science \\
  Canterbury, UK\\
  m.iliopoulou\imageat{}kent\imagedot{}ac\imagedot{}uk \\
  \url{https://sites.google.com/view/marina-iliopoulou/home}
\end{authorinfo}

\end{dajauthors}

\end{document}